\documentclass[11pt]{article}            

\usepackage{enumitem, amsmath,amssymb, amsthm, mathrsfs, amsfonts, latexsym}
\usepackage{algpseudocode,algorithm,algorithmicx,bm}
\usepackage{xfrac}
\usepackage{mathtools}
\usepackage{framed}
\usepackage[margin = 1in]{geometry}

\usepackage[sort]{natbib}
 \def\newblock{\ }%
 \bibpunct[, ]{[}{]}{,}{n}{}{,}%

\def\transp{^{\rm T}}

\usepackage[colorlinks=true,breaklinks=true,bookmarks=true,urlcolor=blue,
     citecolor=blue,linkcolor=blue,bookmarksopen=false,draft=false]{hyperref}

\newcommand{\R}{\mathbb{R}}
\newcommand{\bdry}{\operatorname{bdry}}
\newcommand{\inter}{\operatorname{int}}
\newcommand{\ip}[2]{\left\langle #1 , #2 \right\rangle}    
\newcommand{\prox}{\operatorname{prox}}

\newcommand{\barw}{\overline{w}}
\newcommand{\barW}{\overline{W}}
\newcommand{\barX}{\overline{X}}
\newcommand{\barXY}{\overline{XY}}

\newcommand{\barU}{\overline{U}}
\newcommand{\barZ}{\overline{Z}}

\newcommand{\dist}{\operatorname{dist}}

\renewcommand{\S}{\mathcal S}
\newcommand{\relint}{\operatorname{relint}}
\newcommand{\supp}{\text{\rm supp}}
\newcommand{\graph}{\text{\rm{graph}}}
\newcommand{\cl}{\text{cl}}
\newcommand{\matr}[1]{\begin{bmatrix} #1 \end{bmatrix}}    
\newcommand{\dom}{\operatorname{dom}}
\newcommand{\argmin}{\operatorname{argmin}}
\newcommand{\argmax}{\operatorname{argmax}}

\newcommand{\dmin}{\displaystyle \min}
\newcommand{\dmax}{\displaystyle \max}

\newtheorem{proposition}{Proposition}
\newtheorem{lemma}{Lemma}
\newtheorem{theorem}{Theorem}
\newtheorem{example}{Example}

\title{Linear convergence of the Douglas-Rachford algorithm via a generic error bound condition}

\author{
Javier Pe\~na\thanks{\scriptsize Tepper School of Business, Carnegie Mellon University, USA, {\tt jfp@andrew.cmu.edu}} 
\and Juan C. Vera\thanks{\scriptsize Tilburg School of Economics and Management, Tilburg University, The Netherlands, {\tt j.c.veralizcano@uvt.nl}}
\and Luis F. Zuluaga\thanks{\scriptsize Department of Industrial and Systems Engineering, Lehigh University, USA, {\tt
luis.zuluaga@lehigh.edu}}}

\begin{document}

\maketitle

\begin{abstract}%
We provide new insight into the convergence properties of the Douglas-Rachford algorithm for the problem
$\min_x \{f(x)+g(x)\}$,
where $f$ and $g$ are convex functions. Our approach relies on and highlights the natural primal-dual symmetry between the above problem and its Fenchel dual $\min_{u} \{ f^*(u) + g_*(u)\}$ where $g_*(u):=g^*(-u)$.  Our main development is to show the linear convergence of the algorithm when a natural error bound condition on the Douglas-Rachford operator holds.  We leverage our error bound condition approach to show and estimate the algorithm's linear rate of convergence for three special classes of problems.  The first one is when~$f$ or~$g$  {\em and} ~$f^*$ or~$g_*$ are strongly convex relative to the primal and dual optimal sets respectively.  The second one is when~$f$ and~$g$ are piecewise linear-quadratic functions.  The third one is when~$f$ and~$g$ are the indicator functions of closed convex cones.  In all three cases the rate of convergence is determined by a suitable  measure of well-posedness of the problem.
In the conic case, if the two closed convex cones are a linear subspace $L$ and $\R^n_+$, we establish the following
stronger {\em finite termination} result: the Douglas-Rachford algorithm identifies the {\em maximum support sets} for $L\cap\R^n_+$ and $L^\perp\cap\R^n_+$ in finitely many steps.
Our developments have straightforward extensions to the more general linearly constrained problem $\min_{x,y} \{f(x) + g( y):Ax + By = b\}$ thereby highlighting a direct and straightforward relationship between the Douglas-Rachford algorithm and the alternating direction method of multipliers (ADMM).
\end{abstract}%

{\small {\noindent \bf Keywords: }
Douglas-Rachford algorithm; ADMM algorithm; Hoffman constants; linear rate of convergence; Fenchel duality}

\maketitle

\section{Introduction.}
We provide new insight into the convergence properties of the {\em Douglas-Rachford} algorithm for the problem
\begin{equation}
\label{eq.primal0}
\min_x \{f(x)+g(x)\}.
\end{equation}
where $f$ and $g$ are convex functions.

The Douglas-Rachford algorithm is based on a fixed point iteration of a canonical mapping associated to~\eqref{eq.primal0} that we refer to as the {\em Douglas-Rachford operator} (see~\eqref{eq:DRoperator}). Our main development (Theorem~\ref{thm.DR}) is to show the linear convergence of the Douglas-Rachford algorithm when the Douglas-Rachford operator satisfies a natural error bound condition (see~\eqref{eq.error.bound.F}). Our results rely on and highlight the following primal-dual symmetry (see Propositions~\ref{prop.DR.facts} and~\ref{prop.DR.symm}): the Douglas-Rachford operator for~\eqref{eq.primal0} is the same as the Douglas-Rachford operator for its {\em Fenchel dual}~\citep[see, e.g.,][]{rockafellar1970convex}
\begin{equation}
\label{eq.dual0}
\max_{u} \{- f^*(u) - g^*(-u)\} \Leftrightarrow
\min_{u}  \{f^*(u) + g_*(u)\}.
\end{equation}
For notational convenience in~\eqref{eq.dual0} and throughout the paper  we write $g_*$ to denote the function $u \mapsto g^*(-u)$. 
In particular, the Douglas-Rachford algorithm implicitly solves both~\eqref{eq.primal0} and~\eqref{eq.dual0} and generates identical iterates if we swap the primal and dual roles of these problems (Proposition~\ref{prop.DR.symm}).

The Douglas-Rachford algorithm was introduced in~\citep{douglas1956numerical} to solve heat equations. It was
initially proposed and studied in~\citep{lions1979splitting} as a {\em splitting} algorithm to find the zero sum of two maximally monotone operators. The seminal work in~\citep{lions1979splitting} has led to a very rich and widespread related literature~\citep[see, e.g.,][just to name a few]{parikh2014proximal, boyd2011distributed, eckstein1992douglas, eckstein1989splitting}. In particular,
the linear convergence of the Douglas-Rachford algorithm was first studied in~\citep{lions1979splitting}, and since then, linear convergence of the algorithm has been proved in different settings~\citep[][Sec. I]{giselsson2016linear}.
Our approach to the linear convergence of the Douglas-Rachford is inspired by linear convergence results for the alternating direction methods of multipliers (ADMM) algorithm, which is known~\citep[see, e.g.,][]{boyd2011distributed} to be equivalent to applying the Douglas-Rachford algorithm to~\eqref{eq.dual0}. In particular, consider the work of \citep[][Lem. 5]{wang2017new} that proves that the linear rate of convergence of the ADMM algorithm for  feasible linear optimization problems can be obtained in terms of an error bound on the distance to the problem's optimality set~\citep[][Lem. 2 and 3]{wang2017new}; and the work of \citep{hong2017linear} who characterize the cases in which an error bound condition~\citep[][Lem. 2.3]{hong2017linear} is satisfied to show the linear convergence of the ADMM algorithm in the cases when the combined function $f^* + g_*$ is not necessarily strongly convex.

In contrast with these previous results, our error bound condition~\eqref{eq.error.bound.F} directly concerns the set of fixed points of the Douglas-Rachford operator (see~\eqref{eq:DRoperator}). The set of fixed points reflects the primal-dual symmetry of the the Douglas-Rachford operator, as it is the sum of the primal and dual optimal sets (Proposition~\ref{prop.DR.facts}). Therefore our error bound incorporates the distance to the solution sets of both problem~\eqref{eq.primal0} and its dual~\eqref{eq.dual0}, which allows us to take full advantage of the algorithm’s symmetry properties. This perspective  enables us to derive sufficient conditions for linear convergence and also a mechanism to estimate the linear rate of convergence in terms of different joint conditioning properties of the  functions $f$ and $g$.

We leverage the error bound condition~\eqref{eq.error.bound.F} to show and estimate the linear rate of convergence of the Douglas-Rachford algorithm for three special classes of problems.  The first class is when one of $f$ or $g$ is strongly convex and also one of $f^*$ or $g_*$ is strongly convex (Theorem~\ref{thm.stronglyConv.bound}).  This result  is inspired by and extends the result in~\citep[][Thm. 1]{giselsson2016linear} that  only considers the case when $f$ is both strongly convex and smooth, that is, when $f$ and $f^*$ are both strongly convex. Furthermore, the strong convexity assumption can be relaxed to a context that is {\em relative} to the optimal solution sets.  In this case the rate of convergence of the Douglas-Rachford algorithm is determined by two measures of well-posedness; namely, the {\em relative strong convexity} constants of $f$ or $g$ and of $f^*$ or $g_*$ (see~\eqref{eq.rel.strong.conv} and Theorem~\ref{thm.stronglyConv.bound}).

The second class in when both $f$ and $g$ are piecewise linear-quadratic (Theorem~\ref{thm:HpieceWiseLin}).  In this case,
the rate of convergence of the algorithm is determined by the maximum of a finite collection of {\em relative Hoffman constants}~\citep{pena2020new}.
This case covers a number of interesting problem classes including linear and quadratic programming, support-vector machines, and ridge and LASSO regression.

The third class is when $f$ and $g$ are the indicator functions of closed convex cones $C\subseteq\R^n$ and $K\subseteq\R^n$ respectively.  In this case the problems~\eqref{eq.primal0} and~\eqref{eq.dual0} correspond to the conic feasibility problem
\[
\text{ find } x \in C\cap K
\]
and its alternative
\[
\text{ find } u \in C^\circ\cap K^*,
\]
where $(\cdot)^\circ$ and $(\cdot)^*$ respectively denote the associated polar and dual cones. We show that in this case the convergence properties of the Douglas-Rachford algorithm are determined by some {\em relative subtransversality} constants of the above two feasibility problems (Theorem~\ref{prop.error.gral}).  In the particular case when $C = L$ is a linear subspace and $K=\R^n_+$, we establish the following stronger {\em finite termination} properties (Theorem~\ref{thm.finite.feas}). If either $L\cap \R^n_{++}\ne \emptyset$ or $L^\perp\cap \R^n_{++}\ne \emptyset$, then the Douglas-Rachford algorithm terminates after finitely many steps.  More generally, for any subspace $L\subseteq \R^n$ the algorithm identifies the maximum support sets for both $L\cap \R^n_+$ and $L^\perp\cap \R^n_+$ after finitely many steps.  The latter result provides a partial answer to the open question posed by Dadush, V\'egh, and Zambelli~\citep{dadush2020rescaling} concerning maximum support sets for the feasibility problems $L\cap \R^n_+$ and $L^\perp\cap \R^n_+$.  Our finite termination result (Theorem~\ref{thm.finite.feas}) is inspired by and related to the pioneering work of Spingarn~\cite{spingarn1985primal} and the more recent work of Bauschke et al.~\cite{bauschke2016slater}.  Indeed,~\citep[][Thm. 3.7]{bauschke2016slater} establishes the finite convergence of the  Douglas-Rachford algorithm in the case when the function $f$ is the indicator function of an affine subspace $A$, the function $g$ is the indicator function of a closed convex set $B$ that is polyhedral at every point in $A \cap \bdry{B}$, and the {\em Slater} condition $A \cap \inter{B} \neq \emptyset$ is satisfied.  The latter work in turn can be seen as a refinement of the earlier result in~\citep[][Thm. 2]{spingarn1985primal} that shows the finite convergence of the method of partial inverses for solving a system of linear inequalities provided the system is strictly feasible.  In contrast to~\cite{bauschke2016slater,spingarn1985primal}, our approach is specialized to the problems $L\cap \R^n_+$ and $L^\perp\cap \R^n_+$.  This in turn yields a finite termination result in full generality that does not require a Slater condition.  Furthermore, although similar in spirit, our proof of finite termination is substantially more succinct than the proofs of finite termination in~\cite{bauschke2016slater,spingarn1985primal}. We also extend these results to the more general case in which $L$ is an affine subspace (Proposition~\ref{prop.fixed.affine}).

All of our developments have a straightforward extension  (Proposition~\ref{prop.DR.gral}) to the more general, linearly constrained problem
\begin{align*}
\min_{x,y} \; & f(x) + g( y) \\
& Ax + By = b,
\end{align*}
after noting that the latter can be reformulated as $\min_z \{\tilde f(z) + \tilde g(z)\}$ for the functions $\tilde f(z) = \min_x \{f(x):Ax=z\}$ and $\tilde g(z):=\min_y \{g(y):b-By=z\}$.  We show that this extended version of the Douglas-Rachford algorithm is indeed {\em equivalent} to  the alternating direction method of multipliers (ADMM) (Proposition~\ref{prop.DR.ADMM.equiv}).  The latter fact provides an interesting complement to the well-known equivalence between the Douglas-Rachford algorithm applied to the dual problem~\eqref{eq.dual0} and the ADMM algorithm applied to the primal problem~\eqref{eq.primal0}~\citep[see, e.g.,][]{boyd2011distributed}.

The rest of the article is organized as follows. In Section~\ref{sec:prelim}, we introduce most of the notation and terminology used throughout the article, well-known facts related to optimality conditions of the primal-dual convex minimization problems~\eqref{eq.primal0} and~\eqref{eq.dual0}, as well as the Douglas-Rachford operator and the Douglas-Rachford algorithm for problems~\eqref{eq.primal0} and~\eqref{eq.dual0}.
Section~\ref{sec:DRconv} presents our main development, namely error bound condition~\eqref{eq.error.bound.F}.  This condition
ensures the linear convergence of the Douglas-Rachford algorithm, and
can be used to estimate its rate of convergence in straightforward fashion (Theorem~\ref{thm.DR}). We then apply this result to obtain and estimate the linear rate of convergence of the Douglas-Rachford algorithm for two special cases.  The first one is when one of the functions $f$ or $g$ is strongly convex relative to the primal optimal set and also $f^*$ or $g_*$ is strongly convex relative to the dual optimal set (see Section~\ref{sec:smooth}).  The second one is when $f$ and $g$, or equivalently $f^*$ and $g^*$, are piecewise linear-quadratic functions (see Section~\ref{sec:PQL}). In Section~\ref{sec.intersections}, we consider a third special case; namely,
when $f$ and $g$ are the indicator functions of convex cones.  In this case the problems~\eqref{eq.primal0}  corresponds to the feasibility problem of finding a point in the intersection of the two convex cones
and~\eqref{eq.dual0} corresponds to the alternative feasibility problem defined by the polar and dual cones. In this case, besides estimating the rate of linear convergence of the Douglas-Rachford algorithm, we establish  stronger finite convergence properties of the algorithm. In Section~\ref{sec:genDR}, we discuss
how our results extend in straightforward fashion to the case in which  problem~\eqref{eq.primal0} has additional linear constraints. We finish in Section~\ref{sec:conclusions} with some conclusions and directions of future work.

\section{Preliminaries.}
\label{sec:prelim}
Throughout the article, we use standard notation and results from~\citep{boyd2004convex, rockafellar2009variational, rockafellar1970convex}. Let $\R$ denote the real numbers and $\bar \R = \R \cup \{+\infty\}$ denote the extended real numbers. We denote by $\ip \cdot\cdot$ the standard inner product on $\R^n$ and by $\|\cdot\|$ the corresponding standard (Euclidean) norm.
Also, we will denote by $\Gamma_0 := \Gamma_0(\R^n)$ the class of closed, proper, convex functions $f: \R^n \to \bar \R$. In particular,  $\dom(f):= \{x \in \R^n:f(x) \neq +\infty\}$ is non-empty for all $f\in\Gamma_0$.
Given $f \in \Gamma_0$, the convex conjugate of $f$ is the function $f^* \in \Gamma_0$ defined by $f^*(u) = \sup_x\{\ip ux - f(x)\}$.  For notational convenience, we will sometimes write $f_*(u)$ for $f^*(-u)$.
The subdifferential of $f\in \Gamma_0$ is the set-valued mapping $\partial f:\R^n\rightrightarrows \R^n$
defined  by
\[\partial f (x) = \{u \in \R^n: f(y)-f(x) \ge \ip u{y-x} \text{ for all }y \in \R^n\}.
\]
The construction of the convex conjugate readily implies the following {\em Fenchel-Young inequality} for all $x,u\in \R^n$
\begin{equation}\label{eq.fenchel-young}
 f^*(u) +f(x) \ge \ip{u}{x}.
\end{equation}
Furthermore, equality holds in~\eqref{eq.fenchel-young} if and only if $u\in \partial f(x)$.

Let $\graph(\partial f):=\{(x,u): u \in \partial f(x)\}$.
For all $f \in \Gamma_0$, the subdifferential of $f$ is monotone, that is,
\begin{equation}\label{eq:diffMon}
\ip{v-u}{y-x} \ge 0 \text{ for all } (x,u), (y,v) \in \graph(\partial f).
\end{equation}
The monotonicity of the subdifferential implies Fermat's rule,
\[\argmin_{x \in \R^n} f(x) = \{x \in \R^n: \partial f(x) = 0\}.\]
Further, the subdifferentials of a closed convex function and its conjugate are related by the relationship
\begin{equation}
\label{eq.subdiff}
u \in \partial f(x) \text{ it and only if } x \in \partial f^*(u).
\end{equation}
The proximal mapping of $f$, $\prox_f:\R^n\rightarrow \R^n$, is defined by
\[
\prox_f(x):=\argmin_{y \in \R^n} \left\{f(y) + \frac{1}{2} \|y-x\|^2\right\}.
\]
If $f \in \Gamma_0$ then the proximal mapping is defined everywhere and the following {\em Moreau decomposition} holds for any $x\in \R^n$
\begin{equation}\label{eq.moreau}
\prox_f(x) + \prox_{f^*}(x) = x.
\end{equation}
 The proximal mapping and the subdifferential are closely linked via the relationship
\begin{equation}\label{eq:diff->prox}
u \in \partial f(x) \text{ if and only if } x = \prox_f(x+u),
\end{equation}
which is equivalent to
\begin{equation}\label{eq:prox->diff}
s = \prox_f(x)  \text{ if and only if }  x-s\in \partial f(s).
\end{equation}
Let $x, \hat x \in \R^n$ be given. From~\eqref{eq:prox->diff} and the monotonicity of the subdifferential~\eqref{eq:diffMon} we obtain
$
\ip{x-\prox_f(x)-(\hat x-\prox_f(\hat x))}{\prox_f(x)-\prox_f(\hat x)}\ge 0
$
which in turns implies,
\begin{equation}\label{eq:proxFirmNE}
\ip{\prox_f(x) - \prox_f(\hat x)}{x - \hat x}  \ge  \|\prox_f(x) - \prox_f(\hat x)\|^2.
\end{equation}
That is, the proximal mapping $\prox_f$ is firmly non-expansive.
Applying Cauchy-Schwarz inequality, from~\eqref{eq:proxFirmNE} we obtain
\begin{equation}\label{eq:proxLip}
\|\prox_f(x) - \prox_f(\hat x)\| \le \|x - \hat x\|.
\end{equation}
That is, the proximal mapping $\prox_f$ is also Lipschitz continuous with parameter one.

With these preliminary facts at hand, we can now present the well-known Douglas-Rachford algorithm~\citep{douglas1956numerical,lions1979splitting}. Let $f,g \in \Gamma_0$. Consider the problem
\begin{equation}\label{eq.primal}
\dmin_x \{f(x) + g(x)\} \Leftrightarrow \begin{array}{rl}
\dmin_{x,y}  & f(x) + g(y) \\
& x - y = 0
\end{array}
\end{equation}
and its Fenchel dual~\citep[see, e.g.,][]{rockafellar1970convex}
\begin{equation}\label{eq.dual}
\dmax_u \{-f^*(u) - g^*(-u)\} \Leftrightarrow \dmin_u \{f^*(u) + g_*(u)\} \Leftrightarrow \begin{array}{rl}
\dmin_{u,v}  & f^*(u) + g_*(v) \\
& u - v = 0.
\end{array}
\end{equation}
Define the {\em one-step Douglas-Rachford operator} $F:\R^n\rightarrow \R^n$ as follows
\begin{equation}
\label{eq:DRoperator}
F(w) = w + \prox_g(2\prox_f(w)-w) - \prox_f(w).
\end{equation}
Observe that $F$ in~\eqref{eq:DRoperator} is defined everywhere since $f,g \in \Gamma_0$. The following proposition states some fundamental properties of the Douglas-Rachford operator.

\begin{proposition}\label{prop.DR.facts}  Let $f,g \in \Gamma_0$.  Consider the problems~\eqref{eq.primal},~\eqref{eq.dual}, and the  Douglas-Rachford operator $F$ defined via~\eqref{eq:DRoperator}. The optimal values of~\eqref{eq.primal} and~\eqref{eq.dual} are the same and they are both attained if and only if the set of fixed points $\barW = \{w \in \R^n: w = F(w)\}$ of $F$ is nonempty.  When this is the case, the following correspondences between the optimal solution sets $\barX := \argmin_x\{f(x)+g(x)\}, \; \barU:=\argmax_u\{-f^*(u)-g^*(-u)\},$ and the fixed point set $\barW$ of $F$ hold. On the one hand,
\begin{equation}\label{eq.opt.projections}
\barX = \{\prox_f(\bar w): \bar w\in \barW\}
\text{ and } \barU = \{\prox_{f^*}(\bar w): \bar w\in \barW\}.
\end{equation}
On the other hand,
\begin{equation}\label{eq.fixed.all}
\barW = \barX + \barU.
\end{equation}
\end{proposition}
\begin{proof} The Fenchel-Young inequality~\eqref{eq.fenchel-young} implies that the optimal values of~\eqref{eq.primal} and~\eqref{eq.dual} are the same and attained at  $\bar x$ and $\bar u$ if and only if $(\bar x,\bar u)$  solves the optimality conditions
\begin{equation}
\label{eq.optcond.initial}
\bar u \in \partial f(\bar x) \text{ and } -\bar  u \in \partial g(\bar  x).
\end{equation}
Using~\eqref{eq:diff->prox}, these optimality conditions can be equivalently stated as
\[
\bar x = \prox_f(\bar x+ \bar u), \; \bar x = \prox_g(\bar x-\bar u).
\]
The latter conditions in turn can be rewritten as follows
\begin{equation}
\label{eq.optcond}
\bar x = \prox_f(\bar w), \bar u = \prox_{f^*}(\bar w),  \bar w = F(\bar w), \bar w=\bar x+\bar u.
\end{equation}
The if and only if statement in the proposition as well as the identities~\eqref{eq.opt.projections} and~\eqref{eq.fixed.all} readily follow from the equivalence between
\eqref{eq.optcond.initial} and \eqref{eq.optcond}.
\end{proof}

\medskip

Algorithm~\ref{algo.DR} describes the Douglas-Rachford algorithm~\citep{douglas1956numerical,lions1979splitting} for solving problem~\eqref{eq.primal}.  As Proposition~\ref{prop.DR.symm} below shows, Algorithm~\ref{algo.DR} implicitly solves~\eqref{eq.dual} as well.
\begin{algorithm}[H]
  \caption{Douglas-Rachford} \label{algo.DR}
  \begin{algorithmic}[1]
    \State Pick $w_0\in \R^n$
\For{$k=0,1,2,\dots$} \label{step:two}
	\Statex 	\quad $x_{k+1}:=\prox_f(w_k)$
	\Statex 	\quad $y_{k+1}:=\prox_g(2x_{k+1}-w_{k})$
	\Statex     \quad $w_{k+1}:=w_k+y_{k+1}-x_{k+1}$
\EndFor
	\end{algorithmic}
\end{algorithm}
Observe that the last update in Step~\ref{step:two} of Algorithm~\ref{algo.DR} can also be written as the fixed point iteration
\begin{equation}
\label{eq.update}
w_{k+1}:= F(w_k).
\end{equation}
Proposition~\ref{prop.DR.symm} formalizes the following primal-dual symmetry that permeates throughout the paper:  the Douglas-Rachford algorithm implicitly yields the same sequence of iterates if we swap the roles of the problems~\eqref{eq.primal} and~\eqref{eq.dual}.

\begin{proposition}\label{prop.DR.symm} Suppose $f,g\in \Gamma_0$ and $w_0\in \R^n$.
Let $\{(x_k,y_k,w_k): \; k=1,2,\dots\}$ denote the sequence of iterates generated by Algorithm~\ref{algo.DR} applied to~\eqref{eq.primal} starting from $w_0$.  Likewise, let $\{(u_k,v_k,\tilde w_k): \; k=1,2,\dots\}$ denote the sequence of iterates generated by Algorithm~\ref{algo.DR} applied to
\begin{equation}\label{eq.primal.reverse}
\dmax_u \{-f^*(u) - g^*(-u)\}   \Leftrightarrow \dmin_u\{ f^*(u) +  g_*(u) \}\Leftrightarrow \begin{array}{rl}
\dmin_{u,v}  & f^*(u) + g_*(v) \\
& u - v = 0.
\end{array}
\end{equation}
starting from $\tilde w_0 := w_0$. Then for $k=0,1,\dots$
\[
\tilde w_{k+1} = w_{k+1}, \; u_{k+1} = w_{k} - x_{k+1}, \; v_{k+1} = 2y_{k+1}-x_{k+1}+w_k.
\]
\end{proposition}
\begin{proof}
Let $F$ denote the Douglas-Rachford operator~\eqref{eq:DRoperator} for  problem~\eqref{eq.primal} and let $\tilde F$ denote the the Douglas-Rachford for problem~\eqref{eq.primal.reverse}.  We first show that $\tilde F = F$.  Indeed, we have
\[
\begin{array}{ll}
\tilde F(w) &= w + \prox_{g_*}(2\prox_{f^*}(w)-w) - \prox_{f^*}(w) \\
&= w -\prox_{g^*}(w-2\prox_{f^*}(w)) - \prox_{f^*}(w) \\
& = w - (2\prox_f(w)-w-\prox_g(2\prox_f(w)-w)) - (w-\prox_f(w))
\\
&= F(w).
\end{array}
\]
The second step above follows from the construction of $g_*$. The third  one follows from~\eqref{eq.moreau}.
The fourth one follows from~\eqref{eq:DRoperator}. Since $F=\tilde F$ and $\tilde w_0 = w_0$, it follows that $\tilde w_{k+1} = w_{k+1},\; k=0,1,\dots$.
To conclude the proof observe that for $k=0,1,\dots$ the Moreau decomposition identity~\eqref{eq.moreau} implies that
\[
\begin{array}{l}
 u_{k+1} = \prox_{f^*}(w_k) =  w_k-x_{k+1}\\
 v_{k+1} = \prox_{g_*}(2u_{k+1} - \tilde w_{k}) = -\prox_{g^*}(2 x_{k+1}-w_k) =
y_{k+1}-2x_{k+1}+w_k.
\end{array}
\]
\end{proof}

\bigskip

To conclude this section, we state some well-known basic properties of the mapping~$F$.
\begin{lemma}\label{lem:basic} Let $f,g \in \Gamma_0$ and $F$ be defined by~\eqref{eq:DRoperator}.  Then,
\begin{enumerate}[label = (\roman*)]
\item the mapping $F$ is firmly non-expansive; that is, $\ip{F(w) - F(\hat w)}{ w-\hat w} \ge \|F(w) - F(\hat w)\|^2$ for all $w, \hat w \in \R^n$, \label{it:bas1.3}
\item the mapping $F$ is Lipschitz continuous with parameter one; that is,  $\|F(w) - F(\hat w)\| \le \|w - \hat w\|$ for all $w, \hat w \in \R^n$. \label{it:bas1.4}
\end{enumerate}
\end{lemma}

\begin{proof}
\noindent \ref{it:bas1.3}
Let $w$ and $\hat w$ be given. Let $x = \prox_f(w)$ and $\hat x = \prox_f(\hat w)$,
$y = \prox_g(2x-w)$ and $\hat y = \prox_g(2\hat x-\hat w)$.  Thus $F(w) = w + y - x,\; F(\hat w) = \hat w + \hat y - \hat x$, and
 \begin{align*}
& \ip{F(w) - F(\hat w)}{ w-\hat w} - \|F(w) - F(\hat w)\|^2 \\
&\quad =\ip{F(w)-F(\hat w)}{(w - F(w)) - (\hat w-F(\hat w))}=   \ip{F(w)-F(\hat w)}{(x - \hat x) - ( y - \hat y)}\\
&\quad=   \ip{(F(w)-F(\hat w)) - ( y - \hat y)}{x - \hat x} +  \ip{(x - \hat x) - (F(w)-F(\hat w))}{y - \hat y}\\
 &\quad=   \ip{(w - x)-(\hat w - \hat x)}{x- \hat x} + \ip{(2x -w -y)-(2\hat x -\hat w -\hat y)}{y-\hat y}\\
 &\quad  =(\ip{x- \hat x}{w -\hat w} - \|x - \hat x\|^2) + (\ip{y-\hat y}{(2x -w)-(2\hat x -\hat w)} - \|y-\hat y\|^2). \end{align*}
To finish, observe that the last two terms are non-negative, as the proximal mappings $\prox_f$ and $\prox_g$ are firmly non-expansive~\eqref{eq:proxFirmNE} and $x=\prox_f(w), \hat x = \prox_f(\hat w),y=\prox_f(2x-w), \hat y = \prox_f(2\hat x - \hat w)$.
\ref{it:bas1.4} This follows from~\ref{it:bas1.3} after applying Cauchy-Schwarz inequality.
\end{proof}

\section{Linear Convergence of the Douglas-Rachford algorithm.}
\label{sec:DRconv}
This section presents our main development.  Theorem~\ref{thm.DR} shows the linear convergence of Algorithm~\ref{algo.DR} provided
the following {\em error bound condition} holds on a suitable set $S\subseteq \R^n$
 \begin{equation}\label{eq.error.bound.F}
 \text{There exists a finite constant $H>0$ such that } \dist(w,\barW) \le H\cdot\|(I-F)(w)\| \text{ for all } w\in S.  \end{equation}

This error bound approach is inspired by linear convergence results for the ADMM algorithm~\citep[see, e.g.,][]{boyd2011distributed,hong2017linear,wang2017new}.
However, in contrast with these previous approaches, our error bound condition~\eqref{eq.error.bound.F} concerns the set of fixed points of the Douglas-Rachford operator~\eqref{eq:DRoperator}.
 The set of fixed points reflects the primal-dual symmetry of the the Douglas-Rachford operator, as $\barW=\barX+\barU$. Therefore the error bound~\eqref{eq.error.bound.F} takes into account the distance to the solution sets of both problem~\eqref{eq.primal0} and its dual problem~\eqref{eq.dual0}.

\begin{theorem}[Linear convergence of Douglas-Rachford algorithm] \label{thm.DR}
Let $f,g \in \Gamma_0$ and $F$ be defined by~\eqref{eq:DRoperator}.
If $\barW:=\{w\in \R^n: w = F(w)\} \ne \emptyset$ then there exists $\bar w \in \barW$, $\bar x \in \argmin\{f(x) + g(x)\},$ and $\bar u\in \argmax\{-f^*(u)-g^*(-u)\}$ such that the iterates $(x_k,y_k,w_k)$ generated by Algorithm~\ref{algo.DR} satisfy $w_k\rightarrow \bar w$, $x_k \to \bar x$ and $u_{k} := w_{k-1} - x_k \rightarrow \bar u$.
Furthermore,  if the error bound condition~\eqref{eq.error.bound.F} holds on the set $S:=\{w\in \R^n: \dist(w,\barW) \le \dist(w_0,\barW)\}$, then
 $\dist(w_k,\barW) \rightarrow 0$ linearly, more precisely,
\begin{align}\label{eq.linear.conv}
\dist(w_k,\barW)^2 &\le \left(1-\tfrac{1}{H^2}\right) \cdot \dist(w_{k-1},\barW)^2.
\end{align}
Additionally,  $w_k \rightarrow \bar w$,  $x_k \rightarrow \bar x$, and $u_k\rightarrow \bar u$  R-linearly. If $\barW$ is a singleton, then $w_k\rightarrow \bar w$ linearly.

\end{theorem}
\begin{proof} For any $\hat w \in \barW$, Lemma~\ref{lem:basic}\ref{it:bas1.3} applied to $w := w_{k-1}$ implies that $\ip{w_k-\hat w}{w_{k-1} - \hat w} \ge \ip{w_k - \hat w}{w_k - \hat w}$ or equivalently $\ip{w_k-\hat w}{w_{k-1} - w_k} \ge 0$.  Therefore, for all $\hat w \in \barW$
\begin{align}
\|w_{k}-\hat w\|^2 &= \|w_{k-1} - \hat w\|^2 - \|w_k - w_{k-1}\|^2 - 2\ip{w_k-\hat w}{w_{k-1}-w_k}
\notag \\
& \le \|w_{k-1} - \hat w\|^2 - \|w_k - w_{k-1}\|^2. \label{eq.squares}
\end{align}
  From~\eqref{eq.squares} it follows that $\{w_k:k=0,1,\dots\}$ is bounded and $\|w_k - w_{k-1}\|\rightarrow 0$.  Thus, there exists $\bar w \in \R^n$ such that $w_{k_j} \rightarrow \bar w$ for some subsequence
$\{w_{k_j}:j=0,1,\dots\}$ of $\{w_k:k=0,1,\dots\}$.  We next show that indeed $\bar w\in \barW$ and
 $w_k \rightarrow \bar w$.

Since $w_{k_j} \rightarrow \bar w$, the continuity of $F$ implies that $w_{k_j+1} = F(w_{k_j}) \rightarrow F(\bar w)$.  In addition, since $\|w_k - w_{k-1}\|\rightarrow 0$ and $w_{k_j} \rightarrow \bar w$, it follows that $w_{k_j+1} \rightarrow \bar w$.  Therefore $\bar w = F(\bar w)$, that is, $\bar w \in \barW$. Since~\eqref{eq.squares} holds for all $\hat w\in\barW$, in particular it holds for $\hat w = \bar w$ and so $\|w_k-\bar w\|$ is monotonically decreasing.  Since $\|w_{k_j} - \bar w\| \rightarrow 0$ it follows that $\|w_k-\bar w\|\rightarrow 0$, that is, $w_k \rightarrow \bar w.$

Since $\bar w\in \barW$, it follows that $\bar x:=\prox_f(\bar w) \in  \argmin\{f(x)+g(x)\}$ and $\bar u:=\prox_{f^*}(\bar w) \in \argmax\{-f^*(u)-g^*(-u)\}$.
Also, as the proximal mapping is Lipschitz continuous with parameter one~\eqref{eq:proxLip}, it follows that $\|\prox_f(w_k) - \prox_f(\bar w)\| \le \|w_k - \bar w\| \rightarrow 0$. Thus,  $x_{k+1}=\prox_f(w_k) \to  \prox_f(\bar w)= \bar x$ and $u_{k+1} = w_k - x_{k+1} = \prox_{f^*}(w_k)\rightarrow \prox_{f^*}(\bar w)= \bar u$.

Next, we show~\eqref{eq.linear.conv} when~\eqref{eq.error.bound.F} holds on $\{w\in \R^n: \dist(w,\bar W) \le \dist(w_0,\bar W)\}$.  From~\eqref{eq.squares} it follows that the sequence $\{\dist(w_k,\bar W): k=0,1,\dots\}$ is decreasing in $k$.  Thus the error bound condition~\eqref{eq.error.bound.F} on $\{w\in \R^n: \dist(w,\bar W) \le \dist(w_0,\bar W)\}$ implies that for $k=1,2,\dots$
\[
\dist(w_{k-1},\barW) \le H \cdot \|(I-F)(w_{k-1})\| = H \cdot \|w_k - w_{k-1}\|.
\]
Combining this inequality and inequality~\eqref{eq.squares} and letting
$\bar w_k:=\argmin_{w\in\barW}\|w_k-w\|, \; k=0,1,\dots$
 we get
\begin{align}
\|w_k-\bar w_k\|^2 &\le \|w_k-\bar w_{k-1}\|^2 \notag \\
&\le \|w_{k-1}-\bar w_{k-1}\|^2 - \|w_k-w_{k-1}\|^2 \notag \\&\le  \left(1-\tfrac{1}{H^2}\right) \cdot \|w_{k-1}-\bar w_{k-1}\|^2.\label{eq.linear.step}
\end{align}
In particular,~\eqref{eq.linear.conv} follows. If $\barW$ is a singleton then $\bar w_k = \bar w$ for all $k$ and equation~\eqref{eq.linear.step} shows that $w_k \to \bar w$ linearly. Otherwise, to finish the proof it suffices to show
$w_k\rightarrow \bar w$ R-linearly when~\eqref{eq.linear.step} holds.  To that end, observe that~\eqref{eq.linear.step} implies that for $r:=\sqrt{1-\tfrac{1}{H^2}} < 1$ and $k,j=1,2,\dots$
\[
\|\bar w_{k+j}- \bar w_{k+j-1}\| \le \|\bar w_{k+j}- w_{k+j}\| + \|w_{k+j}-\bar w_{k+j-1}\| \le
 2r^{j+1} \|w_{k-1}-\bar w_{k-1}\|.
\]
In addition, $\bar w_k \rightarrow \bar w$ because $w_k\rightarrow \bar w$ and $\|w_k-\bar w_k\| \rightarrow 0$.  Thus
\begin{align*}
\|w_{k}- \bar w\| &\le \|w_k-\bar w_k\| + \sum_{j=1}^\infty \|\bar w_{k+j}- \bar w_{k+j-1}\| \\
&\le  \left(r + 2\sum_{j=1}^\infty r^{j+1} \right)\cdot \|w_{k-1} - \bar w_{k-1}\| \\
& \le \left(2\sum_{j=0}^\infty r^{j+1}\right)\cdot \|w_{k-1} - \bar w_{k-1}\| \\
& = \frac{2r}{1-r}\|w_{k-1} - \bar w_{k-1}\|
\end{align*}
and so $w_k\rightarrow \bar w$ R-linearly since $\|w_{k-1} - \bar w_{k-1}\| = \dist(w_{k-1},\barW)\rightarrow 0$ linearly. 
\end{proof}

\medskip

Theorem~\ref{thm.DR} allows us to recast the problem of proving and estimating the rate of linear convergence of the Douglas-Rachford algorithm as the problem of computing the constant $H$ in the generic error bound condition~\eqref{eq.error.bound.F}. We will exemplify this procedure by  estimating the rate of linear convergence of the Douglas-Rachford algorithm for problems~\eqref{eq.primal} and~\eqref{eq.dual} when the functions~$f$ and~$g$ or their Fenchel duals belong to well-known classes of functions in~$\Gamma_0$. In particular, we next consider two  very relevant cases. First, in Section~\ref{sec:smooth} we will consider the case when one of~$f$ or~$g$ is strongly convex relative to the primal optimal set and also~$f^*$ or~$g_*$ is strongly convex relative to the dual optimal set. Second, in Section~\ref{sec:PQL} we will consider the case when both functions~$f$ and~$g$, or equivalently $f^*$ and~$g_*$, are piecewise linear-quadratic. In both cases we will show the error bound condition~\eqref{eq.error.bound.F} holds and compute an estimate of~$H$ in terms of the properties of~$f$ and~$g$.

In the separate Section~\ref{sec.intersections}, we further specialize our developments to the case when~$f$ and~$g$ are indicator functions of convex cones.  Namely, when~\eqref{eq.primal} and~\eqref{eq.dual} correspond to a pair of alternative feasibility problems associated with finding a point in the intersection of the two convex cones. In this case, we show that the error bound condition~\eqref{eq.error.bound.F} is satisfied with a constant $H$ that is determined by an interesting geometric property of the intersection of the cones.  We also establish
 stronger {\em finite termination}  properties of the Douglas-Rachford algorithm in the particular case when one of the cones is a linear subspace and the other one is the non-negative orthant.

\subsection{Strong convexity and smoothness.}
\label{sec:smooth}
Suppose $\mu > 0$. Recall that a convex function $f$  is $\mu$-strongly convex if $f(\cdot)-\tfrac {\mu}2\|\cdot\|^2$ is convex, or equivalently if $\partial f$ satisfies the following $\mu$-strong monotonicity (compare to~\eqref{eq:diffMon}):
\begin{equation}\label{eq:strConv}
  \ip{v-u}{y-x} \ge \mu \|y - x\|^2 \text{ for all } (x,u), (y,v) \in \graph(\partial f).
\end{equation}
 The dual concept of strong convexity is smoothness.  A convex function $f$  is $L$-smooth if $\tfrac {L}2\|\cdot\|^2 - f(\cdot)$ is convex, or equivalently
\begin{equation}\label{eq:strSmoth}
\ip{v-u}{y-x} \le L \|y - x\|^2   \text{ for all } (x,u), (y,v) \in \graph(\partial f),
\end{equation}
or equivalently
\begin{equation}\label{eq:strSmoth.dual}
\ip{v-u}{y-x} \ge \tfrac{1}{L} \|v - u\|^2   \text{ for all } (x,u), (y,v) \in \graph(\partial f).
\end{equation}
We refer the reader to~\citep[e.g.,][]{zhou2018fenchel} for additional details on these relationships.
From~\eqref{eq:strConv} and~\eqref{eq:strSmoth.dual} it follows that $f$ is $L$-smooth if and only if $f^*$ is $1/L$-strongly convex. To highlight the primal-dual symmetry of our next results, we state them in terms of strong convexity of the function $f,g$ or of their conjugates $f^*,g_*$.

Recently, \citep[][Thm.~1 and~2]{giselsson2016linear}, study the rate of linear convergence of the Douglas-Rachford algorithm in the case when $f$ is strongly convex and smooth, that is, when {\em both} $f$ and $f^*$ are strongly convex.
As Theorem~\ref{thm.stronglyConv.bound} below shows,
our error bound approach (see   \eqref{eq.error.bound.F} and Theorem~\ref{thm.DR}) allows us to extend their results to the case when $f$ or $g$ {\em and} $f^*$ or $g_*$ are  strongly convex relative to the optimal sets of~\eqref{eq.primal} and~\eqref{eq.dual}, for the concept of {\em relative} strong convexity that we  describe next.

 Suppose $f\in \Gamma_0$ and $X\subseteq \dom(\partial f)$ is a nonempty closed convex set.
Let $\Pi_X:\R^n\rightarrow X$ denote the orthogonal projection onto $X$, that is, the mapping defined as follows
\[
y \mapsto \Pi_X(y):=\argmin_{x\in X}\|y-x\|.
\]
We shall say that {\em $ f$ is $\mu$-strongly convex relative to $X$} if for all $(y,v)\in \text{graph}(\partial f)$ and all $u\in \partial f(\Pi_X(y))$
\begin{equation}\label{eq.rel.strong.conv}
\ip{v-u}{y-\Pi_{X}(y)} \ge \mu\|y-\Pi_{X}(y)\|^2.
\end{equation}
It is evident that  if $f$ is $\mu$-strongly convex then
$ f$ is $\mu$-strongly convex relative to any nonempty closed convex set $X\subseteq\dom(\partial f)$.  More generally, Lemma~\ref{lemma.rel.strong} in Section~\ref{sec:genDR} below shows that
if $f$ is of the form $f = h \circ A$ where $h$ is strongly convex and $A$ is a linear mapping, then $f$ is strongly convex relative to $A^{-1}(X)\subseteq \dom(\partial f)$ for any nonempty closed convex set $X\subseteq\dom(\partial h)$.

The above concept of relative strong convexity is closely related to and inspired by recent developments in~\citep{gutman2020condition}.

\begin{theorem}
\label{thm.stronglyConv.bound} Let  $f,g\in \Gamma_0$ be such that~\eqref{eq.primal} and~\eqref{eq.dual} have the same optimal values and they are both attained.  Let $\barX = \argmin_x\{f(x)+g(x)\}$ and $\barU:=\argmin_u\{f^*(u)+g^*(-u)\} = \argmin_u\{f^*(u)+g_*(u)\}$ and $F$ be defined via~\eqref{eq:DRoperator}.
Suppose~$ f$ or~$ g$ is $\mu$-strongly convex relative to $\barX$ for some $\mu > 0$, and $ f^*$ or~$ g_*$ is $\mu^*$-strongly convex relative to $\barU$ for some $\mu^* > 0$.  Then $F$ satisfies the error bound condition~\eqref{eq.error.bound.F} on $S:=\R^n$ for $H = 4(1+ \max(1/\mu,1/\mu^*))$.
\end{theorem}
\begin{proof}
On the one hand, for $w\in \R^n$ let $x,u,y,v$ be as follows
\[
x:=\prox_f(w), u := \prox_{f^*}(w), y := \prox_g(2x-w), v := \prox_{g_*}(w-2x) = -\prox_{g^*}(2x-w).
\]
From~\eqref{eq.moreau} and~\eqref{eq:DRoperator}, it follows that
\begin{equation}
\label{eq.maprel}
w = x+u \text{ and }F(w) = w+y-x = u + y = x+v.
\end{equation}
On the other hand, Proposition~\ref{prop.DR.facts} implies that for any $(\bar x,\bar u) \in\barX \times \barU$ we have
$\bar x = \prox_f(\bar w)$ and $\bar u = \prox_{f^*}(\bar w)$ for $\bar w := \bar x + \bar u \in \barW = \{w: w= F(w)\}.$  Thus from~\eqref{eq.maprel} it follows that
\begin{align}\label{eq.Fw.w}
\ip{F(w)-\bar w}{w-F(w)} &=
\ip{u+y-\bar x - \bar u}{x-y} \notag\\
&=
\ip{u-\bar u}{\bar x-y}+\ip{y-x}{\bar x - y} + \ip{u-\bar u}{x-\bar x}
\notag\\
&= \ip{u+y-x-\bar u}{\bar x-y} + \ip{u-\bar u}{x-\bar x}
\notag \\
& = \ip{v-\bar u}{\bar x-y} + \ip{u-\bar u}{x-\bar x}.
\end{align}
Observe that $(x,u)\in \text{graph}(\partial f)$ or equivalently $(u,x)\in\graph(\partial f^*)$ , and $(y,-v)\in \text{graph}(\partial g)$ or equivalently
$(v,-y)\in\graph(\partial g_*)$.  Also $(\bar x,\bar u)\in \text{graph}(\partial f)$ or equivalently $(\bar u,\bar x)\in \graph(\partial f^*)$, and $(\bar x,-\bar u)\in \text{graph}(\partial g)$ or equivalently $(\bar u,-\bar x)\in\text{graph}(\partial g_*)$.  Thus the convexity of $f$ and $g$, the relative $\mu$-strong convexity of $f$ or $g$, and the relative $\mu^*$-strong convexity of $f^*$ or $g_*$ imply that if we judiciously choose $(\bar x,\bar u) \in \barX\times \barU$ with either $\bar x = \Pi_{\barX}(x)$ or $\bar x = \Pi_{\barX}(y)$ and with either $\bar u = \Pi_{\barU}(u)$ or $\bar u = \Pi_{\barU}(v)$ then
\begin{equation}\label{eq.thmstr.f}
\ip{v-\bar u}{\bar x-y} + \ip{u-\bar u}{x-\bar x}\ge \frac{1}{2}(\mu_1\|x-\bar x\|^2 + \mu_2\|y-\bar x\|^2 + \mu_3\|v-\bar u\|^2 + \mu_4\|u-\bar u\|^2)
\end{equation}
for some $\mu_i \ge 0, \; i=1,2,3,4$ with $\mu_1+\mu_2 \ge \mu$ and $\mu_3+\mu_4 \ge \mu^*$.

Combining~\eqref{eq.Fw.w} and~\eqref{eq.thmstr.f} we obtain
\begin{align*}
&\ip{F(w)-\bar w}{w-F(w)} \\[1ex]
&\ge \frac{1}{2}(\mu_1\|x-\bar x\|^2 + \mu_2\|y-\bar x\|^2 + \mu_3\|v-\bar u\|^2 + \mu_4\|u-\bar u\|^2)\\[1ex]
& =  \frac{1}{2}(\mu_1\|x-\bar{x}\|^2 + \mu_2\|F(w) - w + x - \bar x \|^2 + \mu_3\|F(w)-x-\bar w+\bar x\|^2 + \mu_4\|w - x - \bar w + \bar x\|^2)
\end{align*}
for some $\mu_i \ge 0, \; i=1,2,3,4$ with $\mu_1+\mu_2 \ge \mu$ and $\mu_3+\mu_4 \ge \mu^*$.  Hence by applying the technical Lemma~\ref{lem:techabc} below with $a := F(w)-\bar w, \, b := w-F(w), c := x-\bar x,$ we conclude that
\[
\|w-\bar w\| \le 4\left(1+ \max(1/\mu,1/\mu^*)\right) \|w-F(w)\|.
\]
Therefore the error bound condition~\eqref{eq.error.bound.F} holds on $S:=\R^n$ for $H = 4(1+ \max(1/\mu,1/\mu^*))$.
\end{proof}

\begin{lemma}\label{lem:techabc} Let $a,b,c \in \R^n$, and $\mu_i \ge 0$ for $i=1,\dots,4$ be such that
\[
\ip{a}{b} \ge \frac{1}{2}(\mu_1 \|c\|^2 + \mu_2\|c-b\|^2 + \mu_3\|a-c\|^2 + \mu_4\|a+b-c\|^2).
\]
Then
\begin{equation}\label{eq.lemma}
\|a + b\| \le  4\left(1 +  1/\mu\right)\|b\|
\end{equation}
for $\mu := \min(\mu_1 + \mu_2,\mu_3 + \mu_4)$.
\end{lemma}
\begin{proof} Notice that $\ip {a}{b} \ge 0$ and thus $\|a+b\|^2 \ge \|a-b\|^2$.
We have
\begin{align*}
\ip{a}{b} & \ge \frac{1}{2}(\mu_1 \|c\|^2 + \mu_2\|c-b\|^2 + \mu_3\|a-c\|^2 + \mu_4\|a+b-c\|^2) \\[1ex]
   & \ge \frac{\mu}{2}(\min(\|c\|^2,\|c-b\|^2) +   \min(\|a-c\|^2,\|a+b-c\|^2)\\[1ex]
   & = \frac{\mu}{2} \min(\|c\|^2+\|a-c\|^2,\|c\|^2+\|a+b-c\|^2,\|c-b\|^2+\|a-c\|^2,\|c-b\|^2+\|a+b-c\|^2)\\[1ex]
   & \ge \frac {\mu}4  \min(\|a\|^2,\|a+b\|^2,\|a-b\|^2)\\[1ex]
   & = \frac {\mu}4 \min(\|a\|^2, \|a-b\|^2).
\end{align*}
We now consider two cases separately.  If $\|a\| \le \|a-b\|$, using Cauchy-Schwarz inequality we obtain
$\|a\|\|b\| \ge \ip{a}{b} \ge \tfrac {\mu}4 \|a\|^2$ and so $\|b\| \ge \tfrac {\mu}4 \|a\|$.  The latter inequality in turn implies that
$$\|a + b\| \le \|a\| + \|b\| \le (1 + 4/{\mu})\|b\|\le 4(1 + 1/{\mu})\|b\|$$
and thus~\eqref{eq.lemma} holds in this case.

If $\|a\| \ge \|a-b\|$, we obtain
$\ip{a}{b} \ge \tfrac {\mu}2 \|a-b\|^2 = \tfrac {\mu}4 (\|a+b\|^2 - 4\ip ab)$. Hence $$\|a+b\|^2 \le 4(1 + 1/{\mu})\ip{a}{b} \le 4(1 + 1/{\mu})\ip{a+b}{b} \le 4(1 + 1/{\mu})\|a+b\|\|b\|,$$ and thus~\eqref{eq.lemma} follows in this case as well.
\end{proof}

\subsection{Piecewise linear-quadratic functions.}
\label{sec:PQL}

Another interesting subclass of $\Gamma_0$ (see Section~\ref{sec:prelim}) is the class of {\em piecewise linear-quadratic} (PLQ) functions.  A function in $\Gamma_0$ is PLQ if there exists a finite polyhedral partition of its domain such that in each of its parts the function is quadratic~\citep[see, e.g.,][Sec. 10.E]{rockafellar2009variational}. Further, if in each polyhedral partition the PQL function is affine, the function is  piecewise linear.
The class of PLQ functions is closed under composition with piecewise linear functions and thus
problem~\eqref{eq.primal} restricted to PLQ functions $f$ and $g$ encompasses several important optimization models such as linear optimization, (convex) quadratic optimization, least squares, and some of its variations such as LASSO, elastic net, and support vector machines (SVM).

The class of PLQ functions is also closed under several important convex operators, such as the Moreau envelope and the Fenchel conjugates (see, e.g., \citep{rockafellar2009variational}, \citep[][Prop. 5.1]{lucet2009piecewise}). It can also be seen that the gradient and proximal operators map the class of  PLQ functions to the class of piecewise linear operators (see, e.g., \citep{rockafellar2009variational}, \citep[][Sec. 5]{lucet2009piecewise}).

Since the class of PLQ functions is closed under composition with piecewise linear functions, then when both $f$ and $g$ are PLQ, the mapping $F$ defined via~\eqref{eq:DRoperator} as well as the mapping $I-F$
are piecewise linear mappings.  Thus the generic property of piecewise linear mappings stated in Lemma~\ref{lemma:HpieceWiseLin} below automatically implies that $F$ satisfies the error bound condition~\eqref{eq.error.bound.F} on any set of the form $\{w: \dist(w,\bar W) \le R\}$ when $f$ and $g$ are PLQ.

Lemma~\ref{lemma:HpieceWiseLin} relies on the following kind of {\em relative Hoffman constant}.  Suppose $P\subseteq \R^n$ is a polyhedron and $G:P\rightarrow \R^m$ is affine on $P$ and such that $G^{-1}(0)\cap P \ne \emptyset$.  Then Hoffman's Lemma~\cite{Hoff52,hoffman2003approximate} implies that
\begin{equation}\label{eq.rel.Hoffman}
H(G|P):= \sup_{x\in P \setminus G^{-1}(0)} \frac{\dist(x,G^{-1}(0)\cap P)}{\|G(x)\|} < \infty
\end{equation}
with the convention that $H(G|P) = 0$ when $P\subseteq G^{-1}(0).$
We call $H(G|P)$ the {\em relative Hoffman constant}  of $G$ relative to $P$.  To cover the case $G^{-1}(0)\cap P = \emptyset$ we extend~\eqref{eq.rel.Hoffman} as follows.  For $R > 0$ let
\[
H^R(G|P):= \sup_{x\in P \setminus G^{-1}(0)} \frac{\min\{\dist(x,G^{-1}(0)\cap P),R\}}{\|G(x)\|}
\]
with the convention that $\dist(x,G^{-1}(0)\cap P) = \infty$ when $G^{-1}(0)\cap P = \emptyset$. It is easy to see that
\[
H^R(G|P) = \left\{\begin{array}{rl} H(G|P)<\infty & \text{ when } \; G^{-1}(0)\cap P \ne \emptyset \\ \frac{R}{\min_{w\in P}\|G(w)\|}  < \infty & \text{ when } \; G^{-1}(0)\cap P = \emptyset.\end{array}\right.
\]

We will also rely on the following notation and terminology.  Suppose $P_1,\dots,P_k\subseteq$ are nonempty polyhedra and $G_i:\R^n\rightarrow \R^n, \;i=1,\dots,k$ are affine mappings.  We shall say that $(G_1,P_1),\dots,(G_k,P_k)$ is a {\em compatible collection} if the following two conditions hold:
\begin{itemize}
\item[(i)] $\R^n = P_1\cup\cdots\cup P_k$
\item[(ii)]  $G_i(x) = G_j(x)$ for all $i,j\in \{1,\dots,k\}$ and all $x \in P_i\cap P_j$.
\end{itemize}
When the above two conditions hold, we shall write
$G=\bigcup_{i=1}^k G_i|P_i$ to denote the piecewise linear mapping defined via
$x\in P_i \mapsto G_i(x)$.  Note that every piecewise linear mapping $G$ is of the form $G=\bigcup_{i=1}^k G_i|P_i$ for some compatible collection $(G_1,P_1),\dots,(G_k,P_k)$.

\begin{lemma}\label{lemma:HpieceWiseLin}
Suppose $(G_1,P_1),\dots,(G_k,P_k)$ is a compatible collection and  $G=\bigcup_{i=1}^k G_i|P_i$.
 If  $G^{-1}(0):=\{w \in \R^n: G(w) = 0\} \neq \emptyset$ then for all $w \in \R^n$ with $\dist(w,G^{-1}(0))\le R$ we have
\begin{equation}\label{eq.error.poly}
\dist(w,G^{-1}(0)) \le \left(\max_{j=1,\dots,k} H^R(G_j|P_j) \right)\cdot\|G(w)\|.
\end{equation}
Furthermore, for $R>0$ sufficiently small, the following sharper bound holds
\begin{equation}\label{eq.error.poly.sharper}
\dist(w,G^{-1}(0)) \le \left(\max_{j: P_j\cap G^{-1}(0)\ne \emptyset} H(G_j|P_j) \right)\cdot\|G(w)\|.
\end{equation}
\end{lemma}
\begin{proof} 
The construction of $H^R(G_j|P_j), \; j=1,\dots,k$ readily implies that for all $w \in P_j$ with $\dist(w,G^{-1}(0))\le R$
\[
\dist(w,G^{-1}(0)) \le H^R(G_j|P_j)\cdot \|G(w)\|.
\]
Therefore~\eqref{eq.error.poly} follows.  Furthermore, observe that
\[
\max_{j=1,\dots,k} H^R(G_j|P_j) = \max\left\{\max_{j: P_j\cap G^{-1}(0)\ne \emptyset} H(G_j|P_j), \max_{j: P_j\cap G^{-1}(0)= \emptyset} \frac{R}{\epsilon_j} \right\}
\]
for $\epsilon_j := \min_{w\in P_j}\|G(w)\|$ which is strictly positive whenever $P_j\cap G^{-1}(0) = \emptyset$.  Hence for $R>0$ sufficiently small we have
\[
\max_{j=1,\dots,k} H^R(G_j|P_j) = \max_{j: P_j\cap G^{-1}(0)\ne \emptyset} H(G_j|P_j)
\]
and thus~\eqref{eq.error.poly.sharper} follows.
\end{proof}

The following result is an immediate consequence of Lemma~\ref{lemma:HpieceWiseLin}.

\begin{theorem}\label{thm:HpieceWiseLin} Suppose $f,g$ are PLQ.  Consider the problems~\eqref{eq.primal},~\eqref{eq.dual}, and the  Douglas-Rachford operator $F$ defined via~\eqref{eq:DRoperator}.  if $\barW = \{w: w= F(w)\}\ne \emptyset$ then $I-F =\bigcup_{i=1}^k G_i|P_i$ for a compatible
 collection $(G_1,P_1),\dots,(G_k,P_k)$.  Consequently for all $R>0$ the error bound condition~\eqref{eq.error.bound.F} holds on $S:=\{w\in \R^n:\dist(w,\barW) \le R\}$ for
\[
H = \max_{j=1,\dots,k} H^R(G_j|P_j)
\]
and for
\[
H = \max_{j: P_j\cap \barW \ne \emptyset} H(G_j|P_j)
\]
when $R>0$ is sufficiently small
\end{theorem}

The following example gives explicit expressions for the compatible collection and relative Hoffman constants in Theorem~\ref{thm:HpieceWiseLin} for a special case.

\begin{example} Suppose $Q\in\R^{n\times n}$ is symmetric positive definite and consider the problem
\begin{equation}\label{eq.quad}
\begin{array}{rl}
\dmin_x & \frac{1}{2}x\transp Q x + c\transp x\\
& x \ge 0.
\end{array}
\end{equation}
This is a particular instance of~\eqref{eq.primal0} for $f(x) = \delta_{\R^n_+}$ and $g(x) = \frac{1}{2}x\transp Q x + c\transp x$. In this case some straightforward calculations  show that
\[
(I-F)(w) = (I+Q)^{-1}(w^- + Qw^++c).
\]
where $w^+ =\max(w,0)$ and $w^- = w-w^+ = \min(w,0)$.  In particular,
\[
\barW = \{w\in \R^n: w^- + Q w^+ + c = 0\}
\]
Thus $I-F = \bigcup_{J\subseteq\{1,\dots,n\}} G_J|P_J$ where $P_J$ and $G_J$ are as follows. For each $J\subseteq\{1,\dots,n\}$ let $P_J := \{w\in\R^n: w_J\ge 0, \; w_{J^c} \le 0\}$ and let $G_J(w):=(I+Q)^{-1}(M_Jw+c)$ where $M_J\in\R^{n\times n}$ is \[
M_J:=\matr{0&0\\0&I_{J^c}} + \matr{Q_J & 0}.
\]
Theorem~\ref{thm:HpieceWiseLin} implies that for~\eqref{eq.quad} the error bound condition~\eqref{eq.error.bound.F} holds on $\{w: \dist(w,\barW)\le R\}$ for
\[
H:=\max_{J\subseteq\{1,\dots,n\}} H^R(G_J|P_J).
\]
It is easy to see that  when $P_J\cap\barW \ne \emptyset$
\begin{equation}\label{eq.HJ}
H^R(G_J|P_J) = H(G_J|P_J) \le \|\left((I+Q)^{-1}M_J\right)^{-1}\| = \|M_J^{-1}(I+Q)\|.
\end{equation}
Similarly, when $P_J\cap\barW = \emptyset$ we can bound $H^R(G_J|P_J)$ in terms of the gap between~\eqref{eq.quad} and its dual
\begin{equation}\label{eq.quad.dual}
\begin{array}{rl}
\dmax_s & -\frac{1}{2}(s+c)\transp Q^{-1} (s+c)\\
& s \le 0.
\end{array}
\end{equation}
To that end, for $J\subseteq \{1,\dots,n\}$ let
\begin{align*}
\text{\sf gap}_J &:= \min\left\{\frac{1}{2}(w^-+Qw^++c)\transp Q^{-1}(w^-+Qw^++c):  w\in P_J\right\}
\\ &
= \min\left\{\frac{1}{2}\|Q^{-1/2}(M_Jw+c)\|^2:  w\in P_J\right\}.
\end{align*}
Observe that $\text{\sf gap}_J$ is the smallest duality gap between~\eqref{eq.quad} and~\eqref{eq.quad.dual} when $x$ and $s$ are restricted to be of the form $x = w^+$ and $s= w^-$ for $w\in P_J$.  In particular, $P_J\cap \barW\ne \emptyset$ if and only if $\text{\sf gap}_J=0.$
Thus when $P_J\cap\barW = \emptyset$ we readily get
$
\|Q^{-1/2}(M_Jw+c)\|^2 \ge 2 \cdot\text{\sf gap}_J > 0
$
for $w\in P_J$ and consequently
\begin{align}\label{eq.epsilonJ}
H^R(G_J|P_J) &=  \max_{w\in P_J}\frac{R}{\|G_J(w)\|}  \notag  \\
&= \max_{w\in P_J} \frac{R}{\|(I+Q)^{-1}Q^{1/2}Q^{-1/2}(M_Jw+c)\|} \notag \\[1ex]
&\le \|Q^{-1/2}(I+Q)\| \cdot \frac{R}{\sqrt{2\cdot \text{\sf gap}_J}}.
\end{align}
Consider the special case when $Q = \text{diag}(d)$ for some $d \in \R^n_{++}$.  For $J\subseteq\{1,\dots,n\}$ we have $P_J\cap \barW\ne \emptyset$ if and only if
$\{j: c_j < 0\} \subseteq J \subseteq\{j:c_j \le 0\}$  and equation~\eqref{eq.HJ} yields
\[
H(G_J|P_J) \le 1 + \max\left\{\max_{j\in J} \frac{1}{d_j},\max_{j\not\in J}d_j\right\},\]
and a more detailed calculation shows that indeed the equality holds.
On the other hand,  $P_J\cap \barW = \emptyset$ if and only if either $c_j > 0$ for some $j\in J$ or $c_j < 0$ for some $j\not \in J$ and in that case~\eqref{eq.epsilonJ} yields
\[
H^R(G_J|P_J) \le \max_{j=1,\dots,n}\frac{1+d_j}{\sqrt{d_j}}\cdot \max\left\{\frac{R\cdot\sqrt{d_j}}{|c_j|}:  (j\in J \text{ and } c_j > 0) \text{ or }
(j\not \in J \text{ and } c_j<0)\right\}.
\]
In this special case the diagonal structure of $Q$ allows us to compute the sharper expression
\[
H^R(G_J|P_J) = \max_{w\in P_J} \frac{R}{\|G_J(w)\|} = \max\left\{\frac{R\cdot(1+d_j)}{|c_j|}: (j\in J \text{ and } c_j > 0) \text{ or }
(j\not \in J \text{ and } c_j<0)\right\}.
\]

\end{example}

\section{Intersection of closed convex sets.}
\label{sec.intersections}
This section develops our linear convergence approach to the Douglas-Rachford algorithm when applied to the following feasibility problem
\begin{equation}\label{eq.intersection}
\text{ find } x \in A\cap B,
\end{equation}
where $A,B\subseteq \R^n$ are non-empty closed convex sets. Problem~\eqref{eq.intersection} can be written in form~\eqref{eq.primal} as
\[
\min_x \delta_A(x) +\delta_B(x),
\]
where $\delta_A$ and $\delta_B$ respectively denote the indicator functions of $A$ and $B$.
Note that if $S \subseteq \R^n$ is a non-empty closed convex set, then $\delta_S \in \Gamma_0$. Further, for any convex set $S \subseteq \R^n$ we have $\prox_{\delta_S}(w) = \Pi_S(w)$, where $\Pi_S$ denotes the orthogonal projection onto $S$. Furthermore, for  any convex set $S \subseteq \R^n$ we have $R_S(w) = 2\Pi_S(w)-w$, where $R_S$ denotes the reflection mapping in $S$.
It thus follows from~\eqref{eq:DRoperator}, that the one-step iteration mapping $F:\R^n \to \R^n$
of the Douglas-Rachford algorithm applied to~\eqref{eq.intersection} is
\begin{equation}
\label{eq:opersetint}
F(w) = \frac{1}{2}(w + R_B(R_A(w))).
\end{equation}

We next establish the error bound condition~\eqref{eq.error.bound.F}
for $F$ as in~\eqref{eq:opersetint} and thereby the linear convergence of the Douglas-Rachford algorithm for the feasibility problem~\eqref{eq.intersection} for three special cases.
First, we consider the case when $A$ and $B$ are non-empty closed convex cones (see Section~\ref{sec:justcones}). Second, we strengthen these results to the case when $A$ is a linear subspace and $B = \R^n_+$ (see Section~\ref{sec.linear.orthant}). Third, we derive analogous results for the case in which $A$ is an affine subspace, and $B = \R^n_+$ (see Section~\ref{sec:affine}).

We will rely on the following  measure of well-posedness of problem~\eqref{eq.intersection}.  Suppose
$A,B\subseteq \R^n$ are closed convex sets such that $A\cap B\ne \emptyset.$
We shall say that $B$ is {\em substransversal relative to} $A$  if
\begin{equation}\label{eq.Hoffman.set}
\S(B|A):=\sup_{x\in A\setminus B}\frac{\dist(x,A\cap B)}{\dist(x,B)} < \infty
\end{equation}
with the convention that $\S(B|A) = 0$ when $A\subseteq B$.
When $\S(B|A)<\infty$, we call $\S(A|B)$ the {\em modulus of subtransversality of $B$ relative to $A$.}  The above terminology is motivated by the concept of {\em subtransversality} as discussed in~\cite[Chapter 7]{ioffe2017variational}.  We will rely on the next proposition that gives generic sufficient conditions for $\S(B|A)$ to be finite.

\begin{proposition}\label{prop.rel.Hoffman.set}
Suppose $A,B\subseteq \R^n$ are closed convex sets such that $A\cap B\ne\emptyset$.  Then $\S(B|A)<\infty$ in the following cases:
\begin{enumerate}[label=(\roman*)]
\item If $A$ and $B$ are polyhedra.
\label{item:oneHoffman}
\item If $A$ and $B$ are cones and $\relint(A)\cap\relint(B) \ne \emptyset$.
\label{item:twoHoffman}
\item If $A$ and $B$ are cones and $A\cap B = \{0\}$.
\end{enumerate}
\end{proposition}
\begin{proof}
We prove each of the cases next.
\begin{enumerate}[label=(\roman*), wide, align=left, leftmargin=0pt, labelindent=0pt]
\item
This is an immediate consequence of Hoffman's Lemma~\cite{Hoff52,hoffman2003approximate}.
\item Since $A$ and $B$ are cones and $\relint(A)\cap\relint(B) \ne \emptyset$, it follows that
\[
A-B:=\{x-y: (x,y)\in A\times B\}
\]
is a linear subspace. Consequently there exists $\rho > 0$ such that
\[
z\in A-B, \|z\|\le \rho \Rightarrow z\in \{x-y: (x,y)\in A\times B, \|x\|\le 1\}.
\]
To finish, it suffices to show that $\dist(x,A\cap B) \le \dist(x,B)/\rho$ for all $x\in A\setminus B$.  Suppose $x\in A\setminus B$ and let $x^B:=P_B(x)$ and $z = x-x^B$.  Observe that $\|-z\| = \|x-x^B\| = \dist(x,B)$ and thus there exists $(\tilde x,\tilde y)\in A\times B$ such that $-z = \tilde x - \tilde y$ and $\|\tilde x\| \le \dist(x,B)/\rho$.  Hence $0 = z-z = x-x^B + \tilde x - \tilde y= (x+\tilde x) - (x^B+\tilde y)$ with $(x+\tilde x,x^B+\tilde y)\in A\times B$.  Therefore $x+\tilde x \in A\cap B$ and $\|\tilde x\| \le \dist(x,B)/\rho$.  Thus $\dist(x,A\cap B) \le \|\tilde x\| \le \dist(x,B)/\rho$.

\item Since $A\cap B = \{0\}$ it follows that
\[
\rho:=\min_{x\in A: \|x\|=1} \dist(x,B) > 0.
\] Therefore for all $x\in A$
\[
\dist(x,A\cap B) = \dist(x,\{0\}) = \|x\| \le \frac{\dist(x,B)}{\rho}.
\]
and consequently $\S(B|A) \le 1/\rho$.

\end{enumerate}
\end{proof}

\subsection{Intersection of closed convex cones.}
\label{sec:justcones}
We next discuss the case when the two sets in~\eqref{eq.intersection} are closed convex cones $C,K \subseteq \R^n$. Consider the feasibility problem
\begin{equation}\label{eq.feas.primal.gral}
\text{ find } x \in C\cap K.
\end{equation}
This problem can be rewritten as
\begin{equation}
\label{eq:intercone}
\min_x  \delta_C(x) +  \delta_K(x).
\end{equation}
In what follows, for any cone $S \subseteq \R^n$, the set $S^* := \{x \in \R^n: \langle x, x' \rangle \ge 0 \text{ for all $x' \in S$}\}$ denotes the dual cone of $S$, and $S^\circ = -S^*$ denotes the polar of $S$. Then, the Fenchel dual of~\eqref{eq:intercone} is
\[
\max_u -\delta_C^*(u) - \delta_K^*(-u)\Leftrightarrow \min_u \delta_{C^\circ}(u) + \delta_{K^*}(u),
\]
which is equivalent to the following alternative feasibility problem
\begin{equation}\label{eq.feas.dual.gral}
\text{ find } u \in C^\circ\cap K^*.
\end{equation}
The primal-dual symmetry between problems~\eqref{eq.feas.primal.gral} and~\eqref{eq.feas.dual.gral} and the primal-dual of the Douglas-Rachford algorithm (Proposition~\ref{prop.DR.symm}) is reflected in the properties and results developed below.

Each iteration of the Douglas-Rachford algorithm applied to the pair of feasibility problems~\eqref{eq.feas.primal.gral} and~\eqref{eq.feas.dual.gral} performs a reflection in $C$ and a reflection in $K$.  This sequence of operations is equivalent to a reflection in $C^\circ$ and a reflection in $K^*$.  All of these reflections are defined by projections onto the pairs $(C,C^\circ)$ and  $(K,K^\circ)$ as we detail next.

For $z\in \R^n$ let $z^C:=\Pi_C(z)$ and $z^{C^\circ}:=\Pi_{C^\circ}(z)$.   From~\eqref{eq.moreau}, it follows that for all $z\in \R^n$ we have
\begin{equation}
\label{eq:conedecomp}
z = z^C + z^{C^\circ}  = z^{C^*} + z^{-C}
\end{equation}
and thus, the reflection of $z$ in $C$ is $R_C(z) = z^C - z^{C^\circ}$.  Likewise for the cone $K$.

The Douglas-Rachford algorithm applied to the above problems \eqref{eq.feas.primal.gral} and \eqref{eq.feas.dual.gral} can be written as
\[
w_{k+1} = F(w_k),
\]
where it follows from \eqref{eq:opersetint}, that the one-step iteration mapping $F:\R^n \rightarrow \R^n$ is given by
\begin{equation}\label{eq.F.gral}
F(w) = \frac 12 (w+R_K(R_C(w)) = \frac 12
(w + (w^C-w^{C^\circ})^K-(w^C-w^{C^\circ})^{K^\circ}).
\end{equation}

\begin{proposition}\label{prop.fixed.gral}
Let $C,K \subseteq \R^n$ be closed convex cones. Then
the set of fixed points of the mapping $F$ in~\eqref{eq.F.gral} is
\begin{equation}\label{eq.fixed.set}\barW = (C \cap K) + (C^\circ \cap K^*).\end{equation}
In particular, $\barW = C \cap K$ if $\relint(C) \cap \relint(K) \ne \emptyset$ and $\barW = C^\circ \cap K^*$ if $\relint(C^\circ) \cap \relint(K^*) \ne \emptyset$.
\end{proposition}
\begin{proof}
Identity~\eqref{eq.fixed.set} readily follows from~\eqref{eq.fixed.all} because $C\cap K = \argmin_x \{\delta_C(x)+\delta_K(x)\}$ and $C^{\circ}\cap K^*=\argmin_u \{\delta_{C^\circ}(u)+\delta_{K^*}(u)\}= \argmax_u \{-\delta_{C}^*(u)-\delta^*_{K}(-u)\}$. The second statement follows from~\eqref{eq.fixed.set} and the fact that $\relint(C)\cap \relint{(K)} \ne \emptyset \Leftrightarrow C^\circ\cap K^* = \{0\}$, the latter in turn being a consequence of the Proper Separation Theorem~\citep[][Thm. 11.3]{rockafellar1970convex}.
\end{proof}

We next show that the error bound~\eqref{eq.error.bound.F} holds for the mapping  $F$ in~\eqref{eq.F.gral} provided that $\max\{\S(K|C),\S( K^* | C^\circ)\} < \infty$. Proposition~\ref{prop.rel.Hoffman.set}  implies that the latter condition holds when $C$ and $K$ are polyhedral, $\relint(C) \cap \relint(K) \ne \emptyset$, or $\relint(C^\circ) \cap \relint(K^*) \ne \emptyset$.

\begin{theorem}\label{prop.error.gral}
Let $C,K \subseteq \R^n$ be closed convex cones.
If $\max\{\S(K|C),\S(K^* | C^\circ)\} < \infty$  then
the mapping $F$ in~\eqref{eq.F.gral} satisfies the error bound condition~\eqref{eq.error.bound.F} on $S:=\R^n$ with $H = \S(K|C) + \S(K^* | C^\circ)$.
\end{theorem}
\begin{proof}
Suppose $w\in \R^n$.  Then
\begin{align*}
\dist(w,\barW) &= \dist(w^C+w^{C^\circ},(C \cap K) + (C^\circ\cap K^*))\\
&\le \dist(w^C,C \cap K) + \dist(w^{C^\circ},C^\circ\cap K^*) \\
& \le \S(K | C)\|w^C - (w^C)^K\| + \S( K^* | C^\circ)\|w^{C^\circ} - (w^{C^\circ})^{K^*}\|\\
&= \S(K | C)\cdot\|(w^C)^{K^\circ}\| + \S( K^* | C^\circ)\cdot\|(w^{C^\circ})^{-K}\|  & \text{(using~\eqref{eq:conedecomp})}\\
&\le (\S(K|C) +\S(K^* | C^\circ)) \cdot \max\{\|(w^C)^{K^\circ}\|,\|(w^{C^\circ})^{-K}\|\}\\
&= H\cdot\max\{\|(w^C)^{K^\circ}\|,\|(w^{C^\circ})^{-K}\|\}.
\end{align*}
To finish, it is enough to show that for all $w \in \R^n$
\begin{equation}\label{eq.to.finish}
\|w-F(w)\| \ge  \max\{\|(w^C)^{K^\circ}\|,\|(w^{C^\circ})^{-K}\|\}.
\end{equation}
To that end, use the identities (recall~\eqref{eq:conedecomp}) $w = w^C + w^{C^\circ}$ and
$(w^C-w^{C^\circ})^K + (w^C-w^{C^\circ})^{K^\circ} = w^C-w^{C^\circ}$
to write $w-F(w)$ in the following two different ways
\[
w-F(w) = w^C - (w^C-w^{C^\circ})^K = w^{C^\circ} - (w^{C^\circ}-w^C)^{K^*}.
\]
The  construction of the projection mappings implies that
\[
\|w-F(w)\| = \|w^C - (w^C-w^{C^\circ})^K\| \ge \|w^C - (w^C)^K\| = \|(w^C)^{K^\circ}\|
\]
and
\[
\|w-F(w)\| = \|w^{C^\circ} - (w^{C^\circ}-w^C)^{K^*}\| \ge \|w^{C^\circ}-(w^{C^\circ})^{K^*}\| = \|(w^{C^\circ})^{-K}\|.
\]
Thus~\eqref{eq.to.finish} follows.
\end{proof}

\begin{proposition}\label{prop.nonzero.gral}
Let $C,K \subseteq \R^n$ be closed convex cones. Suppose the Douglas-Rachford algorithm is applied to the pair of feasibility problems~\eqref{eq.feas.primal.gral} and~\eqref{eq.feas.dual.gral}. Then

\begin{enumerate}[label = (\roman*)]
\item If $K$ is pointed and the algorithm starts from some $w_0 \in \relint(K) \cap \relint(K^*)$
then $w_k \rightarrow \bar w$ for some $\bar w \in \barW$ such that $\bar w^C \in C\cap K \setminus\{0\}$ if $C\cap K\ne\{0\}$ and $\bar w^{C^\circ} \in C^\circ\cap K^* \setminus\{0\}$ if $C^\circ\cap K^*\ne\{0\}.$ \label{prop.nzgral.one}
\item If $C$ is pointed and the algorithm starts from some $w_0 \in \relint(C) \cap \relint(C^*)$
then $w_k \rightarrow \bar w$ for some $\bar w \in \barW$ such that $\bar w^C \in C\cap K \setminus\{0\}$ if $C\cap K\ne\{0\}$ and $\bar w^{C^\circ} \in C^\circ\cap K^* \setminus\{0\}$ if $C^\circ\cap K^*\ne\{0\}.$ \label{prop.nzgral.two}
\end{enumerate}
\end{proposition}
\begin{proof}
We prove statement~\ref{prop.nzgral.one} only,  as statement~\ref{prop.nzgral.two} follows via a similar argument.  First, Theorem~\ref{thm.DR} implies that $w_k\rightarrow \barw \in \barW.$
Second, for any $\hat x \in C\cap K$ we have
\begin{align*}
\ip{\hat x}{w_{k+1}-w_k} &= \frac{1}{2}\ip{\hat x}{(w_k^C-w_k^{C^\circ})^K - (w_k^C-w_k^{C^\circ})^{K^\circ} - (w_k^C+w_k^{C^\circ})} \\[1ex]
&\ge \frac{1}{2}\ip{\hat x}{(w_k^C-w_k^{C^\circ})^K - (w_k^C-w_k^{C^\circ})^{K^\circ} - (w_k^C-w_k^{C^\circ})} \\
&=-\ip{\hat x}{(w_k^C-w_k^{C^\circ})^{K^\circ}} \\
& \ge 0.
\end{align*}
The second step above follows from $\hat x \in C$ and $w_k^{C^\circ}\in C^\circ$.  The third step follows from the identity $w^C_k -w_k^{C^\circ} =
(w^C_k -w_k^{C^\circ})^K + (w^C_k -w_k^{C^\circ})^{K^\circ}$. The fourth step follows from $\hat x \in K$ and $(w^C_k -w_k^{C^\circ})^{K^\circ} \in K^\circ$.
Thus $\ip{\hat x}{w_k}\ge \ip{\hat x}{w_0}$ for all $k=0,1,\dots$.  In particular, if $\hat x \in C\cap K\setminus\{0\}$ then $\ip{\hat x}{w_k^C} \ge \ip{\hat x}{w_k}\ge \ip{\hat x}{w_0} > 0$ where the last inequality holds because $\hat x \in K \setminus \{0\}, w_0 \in \relint(K^*),$ and $K$ is pointed, and it follows that $\ip{\hat x}{w_k^C}$, $k = 1, \dots$, is bounded away from zero.  Therefore  $C\cap K\ne \{0\}$ and $w_0 \in \relint(K) \cap \relint(K^*)$ imply that
 $w_k^C \rightarrow \bar w^C\in C\cap K \setminus\{0\}$.  A symmetric argument shows that if $\relint(C^\circ) \cap \relint(K^*) \ne \emptyset$ then $w_k^{C^\circ} \rightarrow \bar w^{C^\circ} \in C^\circ\cap K^* \setminus\{0\}$.
\end{proof}

\subsection{Intersection of a linear subspace and $\R^n_+$.}
\label{sec.linear.orthant}
Suppose $L \subseteq \R^n$ is a linear subspace and consider the feasibility problem
\begin{equation}\label{eq.feas.primal}
\text{ find } x \in L\cap \R^n_+.
\end{equation}
This problem can be written in form~\eqref{eq.primal} as
\[
\min_x\; \delta_L(x) + \delta_{\R^n_+}(x).
\]
Its Fenchel dual is equivalent to the following alternative feasibility problem
\begin{equation}\label{eq.feas.dual}
\text{ find } u \in L^\perp\cap \R^n_+.
\end{equation}
The pair of problems~\eqref{eq.feas.primal} and~\eqref{eq.feas.dual} are special cases of~\eqref{eq.feas.primal.gral} and~\eqref{eq.feas.dual.gral} when $C=L,\; K=\R^n_+$.  Thus Proposition~\ref{prop.fixed.gral}, Theorem~\ref{prop.error.gral}, and Proposition~\ref{prop.nonzero.gral}\ref{prop.nzgral.one} hold.  Furthermore, Proposition~\ref{prop.rel.Hoffman.set}\ref{item:oneHoffman} implies that for this choice of $C$ and $K$ the condition $\max\{\S(\R^n_+|L),\S(\R^n_+|L^\perp)\} < \infty$ in Theorem~\ref{prop.error.gral} automatically holds for any linear subspace $L\subseteq\R^n_+$.

Theorem~\ref{thm.finite.feas} below shows that the Douglas-Rachford algorithm applied to~\eqref{eq.feas.primal} and~\eqref{eq.feas.dual} also satisfies an interesting finite termination property.   To formalize this statement, we rely on the canonical {\em maximum support sets} $\supp(L)$ and
$\supp(L^\perp)$ defined next.

For $x\in\R^n_+$ let the support of $x$ be defined as  $$\supp(x):=\{i: x^{(i)} > 0\}.$$
Define the {\em maximum support} set of a linear subspace $L\subseteq\R^n$ as follows
\[
\supp(L):=\{i: x^{(i)} > 0 \text{ for some } x\in L\cap \R^n_+\}.
\]
It is easy to see that $\supp(L) = \supp(x)$ for any $x\in L\cap \R^n_+$ such that the set $\supp(x)$ is maximal.
It is also easy to see that $\supp(L)\cap\supp(L^\perp) = \emptyset$.  Furthermore, the classic Goldman-Tucker Theorem~\cite{goldman1956linear} implies that  $\supp(L) \cup  \supp(L^\perp)= \{1,\dots,n\}$.  In other words, the maximum support sets $\supp(L)$ and $\supp(L^\perp)$ partition the index set $\{1,\dots,n\}.$

To ease notation in Theorem~\ref{thm.finite.feas}, we adopt the following notational convention throughout the rest of this section.  For $w \in \R^n$, we let $x:=\Pi_L(w)$ and $u := \Pi_{L^\perp}(w)$ respectively denote the orthogonal projections of~$w$ onto $L$ and $L^\perp$.   Notice that the reflection of $w=x+u$ in~$L$ is $R_L(w) = x-u$. Also, given any vector $z \in \R^n$, the reflection of $z$ in $\R^n_+$ is  $R_{\R^n_+}(z) = |z|$.  Thus, the Douglas-Rachford algorithm applied to~\eqref{eq.feas.primal} and~\eqref{eq.feas.dual} can be written as
\[
w_{k+1} = F(w_k)
\]
where it follows from \eqref{eq:opersetint}, that the Douglas-Rachford operator $F:\R^n \rightarrow \R^n$ is
\begin{equation}\label{eq.F.special}
F(w) = \frac 12(x+u + |x-u|) = \max(x,u).
\end{equation}

\begin{theorem}\label{thm.finite.feas}
Let $L \subseteq \R^n$ be a linear subspace.
Suppose
that the Douglas-Rachford algorithm is applied to the pair of feasibility problems~\eqref{eq.feas.primal} and~\eqref{eq.feas.dual} starting from some $w_0 \in \R^n_{++}$. Then

\begin{enumerate}[label = (\roman*)]
\item The algorithm identifies the maximum support index sets $\supp(L)$ and $\supp(L^\perp)$ after finitely many steps.
That is, there exists $k_0$ such that  both $\supp(L) = \{i: x_k^{(i)} > u_k^{(i)}\}$ and
$\supp(L^\perp) = \{i: x_k^{(i)} < u_k^{(i)}\}$ for $k> k_0$. \label{thm.finite.one}
\item
If $L\cap \R^n_{++} \ne \emptyset$ or $L^\perp\cap\R^n_{++}\ne\emptyset$ then the algorithm terminates after finitely many steps.  That is, there exists $k_0$ such that  either $w_k = \bar w \in L\cap \R^n_{++}$ or $w_k = \bar w \in L^\perp\cap \R^n_{++}$ for $k> k_0$. \label{thm.finite.two}
\end{enumerate}

\end{theorem}
\begin{proof}
First, observe that statement~\ref{thm.finite.two} is a consequence of statement~\ref{thm.finite.one}.  Indeed,  $L\cap \R^n_{++} \ne \emptyset$ precisely when $\supp(L) = \{1,\dots,n\}$.  Hence, if
$L\cap \R^n_{++} \ne \emptyset$,
for $k > k_0$ statement~\ref{thm.finite.one} implies that $x_k > u_k$ and thus $w_{k+1} = \max(x_k,u_k) = x_{k}$, which implies $x_{k+1} = x_k, v_{k+1} = 0$. We have then $w_k = x_k = u_{k_0}$ for all $k > k_0$. A similar argument holds when $L^\perp\cap \R^n_{++} \ne \emptyset$. Therefore it suffices to prove  statement~\ref{thm.finite.one}.

Proposition~\ref{prop.nonzero.gral}\ref{prop.nzgral.one}  implies that $w_k \rightarrow \bar w = \bar x + \bar u \in (L\cap \R^n_+) + (L^\perp\cap\R^n_+)\setminus\{0\}$ with $\supp(\bar x) \subseteq \supp(L)$ and $\supp(\bar u) \subseteq \supp(L^\perp)$. Let $y_k:=w_{k} - w_{k+1}$ and $z_k := w_{k+1}-\bar w$ for $k=0,1,\dots$.  The proof is based on the four claims below.  Statement~\ref{thm.finite.one} follows by combining Claim~3 and Claim~4.  Claim~3 in turn follows from Claim~1 and Claim~2.

\begin{enumerate}[label = {\em Claim~\arabic*:}, wide, align=left, leftmargin=0pt, labelindent=0pt]
\item $\ip{z_k-z_\ell}{y_k-y_\ell} \ge 0$ for all $k,\ell\ge 0$. \label{cl1}
\item There exists $k_0$ such that $\ip{z_k}{y_k} = 0$ for $k\ge k_0$.  \label{cl2}
\item Suppose $n_0 \ge k_0$ and $\lambda_k \ge 0$ for $k=k_0,\dots,n_0$  are such that $\sum_{k=k_0}^{n_0}\lambda_k=1$. Then
$$\ip{\sum_{k=k_0}^{n_0}\lambda_k z_k}{\sum_{k=k_0}^{n_0}\lambda_k y_k}\le 0.$$  \label{cl3}

\item
There exist $n_0 \ge k_0$ and
$\lambda_k > 0$ for $k=k_0,\dots,n_0$  such that for all $\ell \ge k_0$ \[\ip{z_\ell}{\sum_{k=k_0}^{n_0} \lambda_k y_k}\ge 0.\]
Furthermore, if statement~\ref{thm.finite.one} does not hold, then the inequality is strict for some $\ell_0 \ge k_0$. \label{cl4}
\end{enumerate}

To complete the proof, we next prove each of the  four claims above.

\begin{enumerate}[label = {\em Claim~\arabic*:}, wide, align=left, leftmargin=0pt, labelindent=0pt]
\item Observe that
\begin{align*}
\ip{z_k-z_\ell}{y_k-y_\ell} &= \ip{w_{k+1}-w_{\ell+1}}{w_{k} - w_{k+1} - (w_{\ell}-w_{\ell+1})} \\
&= \ip{F(w_{k})-F(w_{\ell})}{w_{k} - F(w_{k}) - (w_{\ell}-F(w_{\ell})}\\
& = \ip{F(w_{k})-F(w_{\ell})}{w_{k} - w_\ell} - \|F(w_{k}) - F(w_{\ell})\|^2.
\end{align*}
Thus Claim 1 follows from Lemma~\ref{lem:basic}(i) applied to $w:=w_{k}$ and $\hat w:=w_{\ell}$.

\item Let $I:=\supp(\bar w) = \supp(\bar x) \cup \supp(\bar u)$.
Proposition~\ref{prop.nonzero.gral}\ref{prop.nzgral.one} implies that $\bar w\ne 0$ and thus $I \neq \emptyset$. Let
$\epsilon:=\min_{i\in I} \bar w^{(i)}>0$, and $k_0$ be such that $\|w_k - \bar w\| < \epsilon/3$ for all $k\ge k_0$.  Suppose $k \ge k_0$. Since orthogonal projections are non-expansive, it follows that
$\|x_k - \bar x\|< \epsilon/3$ and $\|u_k - \bar u\|<\epsilon/3$.  For $i\in \supp(\bar x)\subseteq \supp(\bar w)$ we have $\bar w^{(i)} \ge \epsilon$ and $\bar u^{(i)} =0$. Thus, $w^{(i)}_k > 2\epsilon/3$ and $|u^{(i)}_k| < \epsilon/3$.
Hence for $i\in \supp(\bar x)$ we have
$x^{(i)}_k = w^{(i)}_k - u^{(i)}_k > 2\epsilon/3 - \epsilon/3 = \epsilon/3 > \|u_k - \bar u\|\ge |u_k^{(i)} - \bar u^{(i)}| = |u_k^{(i)}|$ and thus $\max(x_k^{(i)},u_k^{(i)}) = x_k^{(i)}$.  Likewise, for  $i\in \supp(\bar u)$ we have  $\max(x_k^{(i)},u_k^{(i)}) = u_k^{(i)}$.  Since $\bar w = \bar x + \bar u$ with $\bar x\in L$ and $\bar u \in L^\perp$, it follows that
\begin{align*}
\ip{\bar w}{w_{k+1} - w_{k}} &= \ip{\bar x + \bar u}{\max(x_k,u_k) - x_{k}-u_k} \\
&= \ip{\bar x}{\max(x_k,u_k) - x_{k}} + \ip{\bar u}{\max(x_k,u_k) - u_{k}}\\&=
\sum_{i\in \supp(\bar x)} \bar x^{(i)}(\max(x_k^{(i)},u_k^{(i)}) - x_k^{(i)}) +
\sum_{i\in \supp(\bar u)} \bar u^{(i)}(\max(x_k^{(i)},u_k^{(i)}) - u_k^{(i)}) \\
&=
0.
\end{align*}
Claim 2 then follows by observing that $\ip{w_{k+1}}{ w_k- w_{k+1}} = \ip{\max(x_k,u_k)}{\min(x_k,u_k)} = \ip{x_k}{u_k}= 0$ and thus $\ip{z_k}{y_k} = \ip{w_{k+1} - \bar w}{w_{k} -w_{k+1}} = 0$.

\item  Let $k_0$ be as in Claim 2.  Claim 1, Claim 2, and some straightforward algebraic manipulations show that if $n_0 \ge k_0$ and $\lambda_k \ge 0$ for $k=k_0,\dots,n_0$  are such that $\sum_{k=k_0}^{n_0}\lambda_k=1$ then
\begin{align*}
0 &= \sum_{k=k_0}^{n_0} \lambda_k \ip{z_k}{y_k} \\
&= \ip{\sum_{k=k_0}^{n_0} \lambda_k z_k}{ \sum_{k=k_0}^{n_0} \lambda_k y_k} + \frac 12 \sum_{k=k_0}^{n_0} \sum_{\ell=k_0}^{n_0} \lambda_k \lambda_\ell \ip{z_k - z_\ell}{y_k-y_\ell} \\
&\ge  \ip{\sum_{k=k_0}^{n_0} \lambda_k z_k}{ \sum_{k=k_0}^{n_0} \lambda_k y_k}.\notag
\end{align*}

\item Suppose $k_0$ is as in Claim 2 and
let $$K:=\left\{\sum_{k=k_0}^{n_0} \lambda_ky_k: n_0\ge k_0 \text{ and } \lambda_k > 0 \text{ for } k=k_0,\dots,n_0\right\}.$$
Since $K$ is a convex cone, by the Proper Separation Theorem~\citep[][Thm. 11.3]{rockafellar1970convex}, we have $\relint(K)\cap \relint(K^*) \ne \emptyset$.  Therefore there exist $n_0 \ge k_0$ and $\lambda_k > 0$ for $k=k_0,\dots,n_0$ such that $\sum_{k=k_0}^{n_0} \lambda_ky_k \in \relint(K^*) \subseteq K^*$.  For $\ell \ge k_0$ we have $z_\ell = w_{\ell+1} - \bar w = \sum_{k=\ell+1}^\infty (w_k - w_{k+1}) = \sum_{k=\ell}^\infty y_k \in \cl(K)$ and thus
\[
\ip{z_\ell} {\sum_{k=k_0}^{n_0} \lambda_ky_k}\ge 0.
\]

Suppose statement (i) does not hold.  Then there exists $\ell_0\ge k_0$
 such that either  $x^{(i)}_{\ell_0+1} < u^{(i)}_{\ell_0+1}$ for some
$i\in \supp(L)$ or $x^{(i)}_{\ell_0+1} > u^{(i)}_{\ell_0+1}$ for some $i\in \supp(L^\perp)$.
Let $\hat x \in L \cap \R^n_{+}$ and $\hat u \in L^\perp \cap \R^n_{+}$ be such that $\supp(\hat x) = \supp(L)$ and $\supp(\hat u) = \supp(L^\perp)$.  Then for
 $\hat w:=-(\hat x + \hat u)$ and $k\ge k_0$ we get
\begin{align}\label{eq.part.b}
\ip{\hat w}{y_k} &=
\frac{1}{2}\ip{\hat x + \hat u} {|x_k-u_k|-(x_k+u_k)} \notag
\\[1ex]
&= \frac{1}{2}\ip{\hat x}{|x_k-u_k|-(x_{k}-u_k)} +
\frac{1}{2}(\ip{\hat u}{|u_k-x_k|-(u_k-x_k)} \notag\\
& = \sum_{i\in \supp(L)} \hat x^{(i)} \max\{u^{(i)}_k - x^{(i)}_k,0\} +
\sum_{i\in \supp(L^\perp)} \hat u^{(i)} \max\{x^{(i)}_k - u^{(i)}_k,0\} \notag\\
&\ge 0,
\end{align}
and the inequality is strict for $k=\ell_0+1$. Therefore,
$$
\ip{\hat w}{z_{\ell_0}} = \sum_{k=\ell_0+1}^\infty\ip{\hat w}{y_k} \ge \ip{\hat w}{y_{\ell_0+1}}> 0.$$
Inequality~\eqref{eq.part.b} also implies that $\hat w \in K^*$.
Hence, since $\sum_{k=k_0}^{n_0} \lambda_ky_k \in \relint(K^*)$, there exists $\epsilon > 0$ such that $\sum_{k=k_0}^{n_0} \lambda_ky_k  - \epsilon \hat w \in K^*$.  Therefore,
\[
\ip{z_{\ell_0}} {\sum_{k=k_0}^{n_0} \lambda_ky_k - \epsilon \hat w}
\ge 0 \Rightarrow \ip{z_{\ell_0}} {\sum_{k=k_0}^{n_0} \lambda_ky_k}
\ge \epsilon \ip{\hat w}{z_{\ell_0}} >0.
\]
\end{enumerate}
\end{proof}

Theorem~\ref{thm.finite.feas}\ref{thm.finite.one} shows that the Douglas-Rachford algorithm
provides a partial answer to the open question posed by Dadush, V\'egh, and Zambelli~\citep[][Sec. 5]{dadush2020rescaling} concerning an algorithm that simultaneously finds the maximum support sets for the feasibility problems $L\cap \R^n_+$ and $L^\perp\cap \R^n_+$.

Further, Theorem~\ref{thm.finite.feas}\ref{thm.finite.two} is inspired by and related to the pioneering work of Spingarn~\cite{spingarn1985primal} and the more recent work of Bauschke et al.~\cite{bauschke2016slater}.  Indeed,~\citep[][Thm. 3.7]{bauschke2016slater} establishes the finite convergence of the  Douglas-Rachford algorithm in the case when the function $f$ is the indicator function of an affine subspace $A$, the function $g$ is the indicator function of a closed convex set $B$ that is polyhedral at every point in $A \cap \bdry{B}$, and the {\em Slater} condition $A \cap \inter{B} \neq \emptyset$ is satisfied.  The latter work in turn can be seen as a refinement of the earlier result in~\citep[][Thm. 2]{spingarn1985primal} that shows the finite convergence of the method of partial inverses for solving a system of linear inequalities provided the system is strictly feasible.  In contrast to~\cite{bauschke2016slater,spingarn1985primal}, Theorem~\ref{thm.finite.feas}\ref{thm.finite.two} is specialized to the problems $L\cap \R^n_+$ and $L^\perp\cap \R^n_+$.  This in turn yields a finite termination result in full generality that does not require a Slater condition.  Furthermore, although similar in spirit, our proof of finite termination is substantially more succinct than the proofs of finite termination in~\cite{bauschke2016slater,spingarn1985primal}. We also extend these results to the more general case in which $L$ is an affine subspace in the next section (Proposition~\ref{prop.fixed.affine}).

\subsection{Intersection of an affine subspace and $\R^n_+$.}
\label{sec:affine}
The previous approach can be extended to the following more general feasibility problem.
Suppose $\tilde L \subseteq \R^{n-1}$ is an affine subspace and consider the feasibility problem
\[
\text{ find } \tilde x \in \tilde L\cap \R^{n-1}_+.
\]
Via homogenization, this problem can be recast as
\begin{equation}\label{eq.affine}
\text{ find } x \in  L\cap \R^{n}_+, \; x^{(1)} \ge 1 \Leftrightarrow \text{ find } x \in  L\cap P
\end{equation}
where
$
P:=[1,\infty)\times \R^{n-1}_+
$ and $L:= \left\{t\cdot\matr{1\\ \tilde x}: \tilde x\in L, t\in \R\right\}.$

Problem~\eqref{eq.affine} can be written in form~\eqref{eq.primal} as
\[
\min_x \; \delta_L(x) + \delta_{P}(x)
\]
Each iteration of the Douglas-Rachford algorithm applied to the problem~\eqref{eq.affine} performs a reflection in $L$ and a reflection in $P$.  For $z  \in \R^n$ the reflection of $z$ in $P$ is
$
|z-e^1| + e^1,
$ where $e^1\in \R^n$ is the vector with first component equal to one and all others equal to zero.
Thus the Douglas-Rachford algorithm applied to~\eqref{eq.affine} can be written as
\[
w_{k+1} = F(w_k)
\]
where it follows from~\eqref{eq:opersetint}, that
the Douglas-Rachford operator $F:\R^n \rightarrow \R^n$ is given by
\begin{equation}\label{eq.F.affine}
F(w) = \frac 12(x+u+ |x-u-e^1| + e^1) = \max(x,u+e^1).
\end{equation}

Straightforward modifications of the proofs of Proposition~\ref{prop.fixed.gral}, Theorem~\ref{prop.error.gral}, and Theorem~\ref{thm.finite.feas} yield the following properties of the Douglas-Rachford algorithm when applied to~\eqref{eq.affine}.

\begin{proposition}\label{prop.fixed.affine} Suppose $L\subseteq\R^n$ is a linear subspace and $P=[1,\infty)\times\R^{n-1}_+$ are such that $L\cap P \ne \emptyset$.
\begin{enumerate}[label = (\roman*)]
\item The set of fixed points of the mapping $F$ in~\eqref{eq.F.affine} is
\begin{equation}\label{eq.fixed.affine}\barW = (L \cap P) + (L^\perp \cap \R^n_+).\end{equation}
In particular, $\barW = L \cap P$ if $L \cap \R^n_{++} \ne \emptyset$.
\item The mapping $F$ in~\eqref{eq.F.affine} satisfies the error bound condition~\eqref{eq.error.bound.F} on $S:= \R^n$ for $H = \S(P|L)+\S(\R^n_+|L^\perp).$
\item Suppose that  the Douglas-Rachford algorithm is applied to~\eqref{eq.affine}  starting from some $w_0 \in P$. If $L\cap \R^n_{++} \ne \emptyset$ then the algorithm terminates after finitely many steps.  That is, there exists $k_0$ such that  $w_k = \bar w \in L\cap P$ for $k\ge k_0$.
\end{enumerate}
\end{proposition}

\section{Linearly Constrained Case.}
\label{sec:genDR}

Consider the more general, linearly constrained problem
\begin{equation}\label{eq.primal.gral}
\begin{array}{rl}
\dmin_{x,y}  & f(x) + g(y) \\
& Ax + By = b.
\end{array}
\end{equation}
and its Fenchel dual
\begin{equation}\label{eq.dual.gral}
\dmax_u \;  - f^*(A^*u) - g^*(B^*u) + \ip{b}{u}
\end{equation}
where $f \in \Gamma_0(\R^n), g\in \Gamma_0(\R^m), b\in \R^k,$ and $A:\R^n\rightarrow \R^k,\,B:\R^m\rightarrow \R^k$ are linear mappings.

We can rewrite these two problems in the unconstrained primal-dual format of problems~\eqref{eq.primal} and~\eqref{eq.dual} as
\begin{equation}\label{eq.primal.gral.again}
\dmin_{z} \tilde f(z) + \tilde g(z)
\end{equation}
and
\begin{equation}\label{eq.dual.gral.again}
\dmax_u \; -\tilde f^*(u) - \tilde g^*(-u)
\end{equation}
where
\begin{equation}\label{eq.tilde.funcs}
\tilde f(z) : = \min_x\{f(x): Ax = z\}
\; \text{ and }
\;
\tilde g(z) : = \min_y\{g(y): b- By = z\}.
\end{equation}
Observe that
\[
z = \prox_{\tilde f}(w) \Leftrightarrow z = Ax_+ \text{ for some } x_+ \in \argmin_x\left\{f(x) + \frac{1}{2}\|Ax-w\|^2\right\},
\]
and similarly,
\[
z = \prox_{\tilde g}(w) \Leftrightarrow z = b - By_+ \text{ for some } y_+ \in
\argmin_y\left\{g(y) + \frac{1}{2}\|b- By-w\|^2\right\}.
\]
Since $f \in \Gamma_0(\R^n)$ and $g\in \Gamma_0(\R^m)$ the above minimizers exist for all $w\in\R^k$.  In particular, the mappings $\prox_{\tilde f}$ and $\prox_{\tilde g}$ are defined everywhere.

Algorithm~\ref{algo.DR} applied to $\min_z \{\tilde f(z) + \tilde g(z)\}$ can be rewritten as Algorithm~\ref{algo.DR.gral} for the problems~\eqref{eq.primal.gral} and~\eqref{eq.dual.gral}.  In this case the Douglas-Rachford operator $F$ can be written as
\begin{equation}\label{eq.DR.gral}
F(w) = w - Ax_+-By_++b
\end{equation}
where
\[
x_+ \in \argmin_x\left\{f(x) + \frac{1}{2}\|Ax-w\|^2\right\}, \; y_+
\in \argmin_y\left\{g(y) + \frac{1}{2}\|b- By-2Ax_++w\|^2\right\}.
\]
\begin{algorithm}[H]
  \caption{Douglas-Rachford general case} \label{algo.DR.gral}
  \begin{algorithmic}[1]
    \State Pick $w_0\in \R^q$
\For{$k=0,1,2,\dots$}
	\Statex 	\quad $x_{k+1}:=\argmin\left\{f(x) + \frac{1}{2}\|Ax - w_k\|^2 \right\}$
	\Statex \quad $y_{k+1}:=\argmin\left\{g(y) + \frac{1}{2}\|b - By- 2Ax_{k+1} + w_k\|^2 \right\}$
	\Statex 	\quad $w_{k+1}:=w_k-Ax_{k+1}-By_{k+1}+b$
\EndFor
	\end{algorithmic}
\end{algorithm}

Proposition~\ref{prop.DR.facts} and Theorem~\ref{thm.DR} extend in a straightforward fashion to this more general context as follows.

\begin{proposition}\label{prop.DR.facts.gral}
Let $f \in \Gamma_0(\R^n), g\in \Gamma_0(\R^m), b\in \R^k,$ and $A:\R^n\rightarrow \R^k, B:\R^m\rightarrow \R^k$ be linear mappings. Consider the problems~\eqref{eq.primal.gral},~\eqref{eq.dual.gral} and
the Douglas-Rachford operator $F$ defined via~\eqref{eq.DR.gral}. The optimal values of~\eqref{eq.primal.gral} and~\eqref{eq.dual.gral} are the same and they are both attained if and only if the set of fixed points $\barW = \{w \in \R^n: w = F(w)\}$ of $F$  is nonempty. When this is the case, the following correspondences between the optimal sets $\barXY := \argmin_{x,y}\{f(x)+g(y):Ax+By=b\}$, $\barU:=\argmax_u\{-f^*(A^*u)-g^*(B^*u)+\ip{b}{u}\},$ and the fixed point set
$\barW$ of $F$ hold. On the one hand,
\begin{align}\label{eq.opt.projections.gral}
\barXY = \{(\bar x,\bar y):\; &\bar x\in \argmin\left\{f(x)+\tfrac{1}{2}\|Ax-\bar w\|^2\right\} \text{ and } \\ \notag
&\bar y\in \argmin\left\{g(y)+\tfrac{1}{2}\|b-By-2A\bar x+\bar w\|^2\right\} \; \text{ for some } \bar w\in \barW\}.
\end{align}
On the other hand,
\begin{equation}\label{eq.fixed.all.gral}
\barW = \barZ + \barU
\end{equation}
for $\barZ = \{A\bar x:(\bar x,\bar y) \in \barXY\} =  \{b - B\bar y:(\bar x,\bar y) \in \barXY\}$.
\end{proposition}
\begin{proof} This is a straightforward  modification of the proof of Proposition~\ref{prop.DR.facts}.
The Fenchel-Young inequality~\eqref{eq.fenchel-young} implies that the optimal values of~\eqref{eq.primal.gral} and~\eqref{eq.dual.gral} are the same and attained at  $(\bar x,\bar y)$ and $\bar u$ if and only if $(\bar x,\bar y,\bar u)$  solves the optimality conditions
\begin{equation}
\label{eq.optcond.initial.gral}
A\bar x+ B\bar y = b, \; A^*\bar u \in \partial f(\bar x), \;  B^*\bar  u \in \partial g(\bar  x).
\end{equation}
These optimality conditions can be equivalently stated as
\[
A\bar x+ B\bar y = b, \; A^*(\bar w - A\bar x) \in \partial f(\bar x), \; B^*(b - B\bar y - 2A\bar x + \bar w ) \in \partial g(\bar  y), \bar w = A\bar x + \bar u.
\]
The latter conditions in turn can be rewritten as follows
\begin{equation}
\label{eq.optcond.gral}
\begin{array}{rcl}
\bar x &\in& \argmin\left\{f(x)+\tfrac{1}{2}\|Ax-\bar w\|^2\right\}\\[1ex]
\bar y&\in& \argmin\left\{g(y)+\tfrac{1}{2}\|b-By-2A\bar x+\bar w\|^2\right\} \\
\bar w &=& F(\bar w)\\
\bar w&=&A\bar x + \bar u.
\end{array}
\end{equation}
The if and only if
statement in the proposition as well as the identities~\eqref{eq.opt.projections.gral} and~\eqref{eq.fixed.all.gral} readily follow from the equivalence between
\eqref{eq.optcond.initial.gral} and \eqref{eq.optcond.gral}.
\end{proof}
\bigskip
A straightforward modification of the proof of Theorem~\ref{thm.DR} yields the following result.

\begin{proposition}\label{prop.DR.gral}
Let $f \in \Gamma_0(\R^n), g\in \Gamma_0(\R^m), b\in \R^k,$ and $A:\R^n\rightarrow \R^k, B:\R^m\rightarrow \R^k$ be linear mappings. Consider the problems~\eqref{eq.primal.gral},~\eqref{eq.dual.gral} and the Douglas-Rachford operator $F$ defined via~\eqref{eq.DR.gral}.  If $\barW:=\{w\in\R^k: w = F(w)\}\ne \emptyset$ then there exists $\bar w \in \barW$, $\bar z \in \overline{Z} :=
\{A\bar x = b-B\bar y: (\bar x,\bar y) \in \argmin_{x,y}\{f(x)+g(y):Ax+By=b\}\}$ and $\bar u \in \barU :=
\argmax_u\{-f^*(A^*u)-g^*(B^*u)+\ip{b}{u}\}$ such that the
iterates $(x_k,y_k,w_k)$ generated by Algorithm~\ref{algo.DR.gral} satisfy
 $w_k\rightarrow \bar w$, $Ax_k \to \bar z$, $b-By_k \to \bar z$, and $u_{k} := w_{k-1}-Ax_k \rightarrow \bar u$. Furthermore,  if $F$ satisfies the error bound condition~\eqref{eq.error.bound.F} on $\{w\in\R^k: \dist(w,\barW)\le \dist(w_0,\barW)\}$, then
 $\dist(w_k,\barW) \rightarrow 0$ linearly, more precisely,
\begin{align}\label{eq.linear.conv.gral}
\dist(w_k,\barW)^2 &\le \left(1-\tfrac{1}{H^2}\right) \cdot \dist(w_{k-1},\barW)^2.
\end{align}
Additionally,  $w_k \rightarrow \bar w$,  $Ax_k \rightarrow \bar z$,  $b-By_k \to \bar z$, and $u_k\rightarrow \bar u$  R-linearly. If $\barW$ is a singleton, then $w_k\rightarrow \bar w$ linearly.
\end{proposition}

Notice that Proposition~\ref{prop.DR.gral} does not include a convergence statement for the primal sequence $(x_k,y_k)$.  This is due to the fact that for each $w_k \in \barW$ the set
$
\argmin\left\{f(x)+\tfrac{1}{2}\|Ax- w_k\|^2\right\}
$
is not necessarily a singleton and neither is the set
$
 \argmin\left\{g(y)+\tfrac{1}{2}\|b-By-2Ax_k+ w_k\|^2\right\}. 
$
It is easy to see that under some additional mild assumptions on $f,g$ and/or $A,B$ these two sets are singletons and thus $(x_k,y_k)\rightarrow (\bar x,\bar y) \in\argmin_{x,y}\{f(x)+g(y):Ax+By=b\}$ .

It is well known~\citep[see, e.g.,][]{boyd2011distributed} that the
alternating direction  method of multipliers (ADMM)
 applied to the unconstrained problem~\eqref{eq.primal} is equivalent to applying the Douglas-Rachford algorithm to its Fenchel dual problem~\eqref{eq.dual}. To conclude this section,
we show that the Douglas-Rachford algorithm is indeed {\em equivalent} to the ADMM in the more general linearly constrained setting. Namely, we next show that Algorithm~\ref{algo.DR.gral} is equivalent to the ADMM
 as described in Algorithm~\ref{algo.ADMM.gral}, when applied to the linearly constrained problem~\eqref{eq.primal.gral}.  For ease of notation, the description in  Algorithm~\ref{algo.ADMM.gral} includes a lag in the index of the sequence $\tilde y_k, k=0,1,\dots$.

\begin{algorithm}[H]
  \caption{ADMM} \label{algo.ADMM.gral}
  \begin{algorithmic}[1]
    \State Pick $\tilde x_0\in \dom(f), \tilde u_0\in\R^n$
\For{$k=0,1,2,\dots$}
	\Statex 	\quad $\tilde y_{k}:=\argmin\{g(y) - \ip{\tilde u_k}{A\tilde x_k+By-b} + \frac{1}{2}\|A\tilde x_k+By-b\|^2\}$ 
	\Statex 	\quad $\tilde x_{k+1}:=
	\argmin\{f(x) - \ip{\tilde u_k}{Ax+B\tilde y_{k}-b} + \frac{1}{2}\|Ax+B\tilde y_{k}-b\|^2\}$
	\Statex \quad $\tilde u_{k+1}:=\tilde u_{k}- A\tilde x_{k+1}-B\tilde y_{k}+b$
\EndFor
	\end{algorithmic}
\end{algorithm}

\begin{proposition}\label{prop.DR.ADMM.equiv} The sequence
$\{(\tilde x_{k+1},\tilde y_{k+1},\tilde u_{k+1}), \; k =0,1,\dots\}$ generated by Algorithm~\ref{algo.ADMM.gral} starting from $(\tilde x_0,\tilde u_0)$ is identical to the sequence
$\{(x_{k+1},y_{k+1},w_k-Ax_{k+1}), \; k =0,1,\dots\}$ generated by Algorithm~\ref{algo.DR.gral} starting from $w_0:=b-B\tilde y_0+\tilde u_0$.
\end{proposition}
\begin{proof}
For ease of exposition, let $u_{k+1} := w_k - Ax_{k+1},\; k=0,1,\dots$.  The identity $w_{k+1} = w_k - Ax_{k+1} -By_{k+1} + b$ in Algorithm~\ref{algo.DR.gral} implies that $w_{k+1} = u_{k+1} -By_{k+1} +b$ for $k=0,1,\dots$.  Thus by letting $u_0:=\tilde u_0$ and $y_0:=\tilde y_0$ we have
$$w_k = u_k - By_k+b\; \text{ for } \; k=0,1,\dots.$$
Therefore after some
straightforward algebra, the update rules in Algorithm~\ref{algo.DR.gral} can be rewritten as
follows
\begin{align*}
x_{k+1}&=\argmin\left\{f(x) + \tfrac{1}{2}\|Ax - u_{k}+By_{k}-b\|^2 \right\}
\\[1ex]&= \argmin\left\{f(x) -\ip{u_{k}}{Ax+By_{k}-b}+ \tfrac{1}{2}\|Ax +By_{k}-b\|^2 \right\} \\[1ex]
u_{k+1} &= u_{k} -Ax_{k+1}-By_{k}+b\\[1ex]
y_{k+1} &= \argmin\left\{g(y) + \tfrac{1}{2}\|b - By- Ax_{k+1} + u_{k+1}\|^2 \right\}\\[1ex]
&=\argmin\left\{g(y) - \ip{u_{k+1}}{Ax_{k+1}+By-b}+ \tfrac{1}{2}\|Ax_{k+1}+By- b\|^2 \right\}.
\end{align*}
Since $u_0 = \tilde u_0$ and $y_0 = \tilde y_0$, it follows by induction and the update rules in Algorithm~\ref{algo.ADMM.gral} that for $k=0,1,\ldots$
\[\tilde x_{k+1} =  x_{k+1}, \; \tilde u_{k+1} = u_{k+1}, \; \tilde y_{k+1} = y_{k+1}.\]
\end{proof}

The linear convergence results from Section~\ref{sec:DRconv} have straightforward extensions to the context of this section.  More precisely, if both $f$ and $g$ in~\eqref{eq.primal.gral} are PLQ, then so are $\tilde f$ and $\tilde g$ as defined in~\eqref{eq.tilde.funcs}.  Thus
Theorem~\ref{thm:HpieceWiseLin} implies that the Douglas-Rachford operator~\eqref{eq.DR.gral} satisfies
the error bound condition~\eqref{eq.error.bound.F} on any set of the form $\{w\in \R^k:\dist(w,\barW)\le R\}$.

In analogous fashion,  Lemma~\ref{lemma.rel.strong} below
shows that if $f$ or $g$ is strongly convex then so is $\tilde f$ or $\tilde g$.   Lemma~\ref{lemma.rel.strong} also shows that if $f^*$ or $g^*$ is strongly convex then $\tilde f^*$ or $\tilde g^*$ is strongly convex relative to suitable subsets of their domains.  Therefore if $f$ or $g$  and $f^*$ or $g^*$ are strongly convex and in addition~\eqref{eq.primal.gral} and~\eqref{eq.dual.gral} have the same optimal values and they are both attained then Lemma~\ref{lemma.rel.strong} and~Theorem~\ref{thm.stronglyConv.bound} imply that
the Douglas-Rachford operator~\eqref{eq.DR.gral} satisfies
the error bound condition~\eqref{eq.error.bound.F} on $\R^k$.

\begin{lemma}\label{lemma.rel.strong} Suppose $h\in \Gamma_0(\R^n)$ and $A:\R^m\rightarrow \R^n$ is a non-zero linear mapping such that $\dom(h \circ A)\ne \emptyset.$

\begin{enumerate}[label = (\roman*)]
\item If $h$ is $L$-smooth for some $L > 0$ then $h\circ A$ is $L\cdot \|A\|^2$-smooth. \label{lemma.i}
\item If $h$ is $\mu$-strongly convex for some $\mu > 0$ then $h\circ A$ is $\mu\cdot
(\sigma_{\min}^+(A))^2$-strongly convex relative to $A^{-1}(X) \subseteq \dom(\partial (h\circ A))$ for any nonempty closed convex $X \subseteq \dom(\partial h)$ where $\sigma_{\min}^+(A)$ denotes the smallest positive singular value of $A$, that is,
\[
\sigma_{\min}^+(A):=\min\{\|Az\|: \|z\| =1 \text{ and } z \in \ker(A)^\perp \}.
\]
\label{lemma.ii}
\end{enumerate}
\end{lemma}
\begin{proof} Recall that $\partial (h\circ A)(x) = A^* \partial h(Ax)$ for all $x\in \dom (\partial (h\circ A))$.

\ref{lemma.i}
Suppose $(x,u), (y,v) \in \graph(\partial (h\circ A))$.  Then $u = A^*w$ and $v = A^*z$ for some $w\in \partial h(Ax)$ and $z\in \partial h(Ay)$. Since $h$ is $L$-smooth we have
\[
\ip{v-u}{y-x} = \ip{A^*z-A^*w}{y-x} = \ip{z-w}{Ay-Ax)} \le L\cdot \|A(y-x)\|^2 \le L \cdot \|A\|^2 \cdot \|y-x\|^2.
\]
Since this holds for all $(x,u), (y,v) \in \graph(\partial (h\circ A))$ it follows that $h\circ A$ is $L\cdot \|A\|^2$-smooth.

\ref{lemma.ii} Suppose $X \subseteq \dom(\partial h)$ is a  nonempty closed convex set and $(y,v) \in \graph(\partial (h\circ A))$.  Thus $v = A^*z$ for some $z\in \partial h(Ay)$. Let $\bar y:= \Pi_{A^{-1}(X)}(y)$. Since $h$ is $\mu$-strongly convex, for all $w \in \partial h(A\bar y)$ we have
\[
\ip{A^*z-A^*w}{y-\bar y} =  \ip{z-w}{Ay-A\bar y} \ge \mu\cdot\|A(y-\bar y)\|^2.
\]
Since $\bar y = \Pi_{A^{-1}(X)}(y)$, it follows that $y-\bar y \in N_{A^{-1}(X)}(\bar y)$.  In particular, $y-\bar y \in \ker(A)^\perp$ because $A(\bar y \pm z) = A\bar y \in X$ and thus $\bar y\pm z \in A^{-1}(X)$ for all $z\in \ker(A)$.  Therefore for all
$w \in \partial h(A\bar y)$ we have
\[
\ip{A^*z-A^*w}{y-\bar y} \ge \mu \cdot (\sigma_{\min}^+(A))^2 \cdot\|y-\bar y\|^2.
\]
Since this holds for all $w\in \partial h(A\bar y)$, and we have $v = A^*z$ and  $\partial (h\circ A)(\bar y) = A^*\partial h(A\bar y)$, it follows that
\[
\ip{v-u}{y-\bar y} \ge \mu \cdot (\sigma_{\min}^+(A))^2 \cdot\|y-\bar y\|^2
\]
for all $u\in \partial (h\circ A)(\bar y)$.  Therefore $h\circ A$ is $\mu\cdot
(\sigma_{\min}^+(A))^2$-strongly convex relative to $A^{-1}(X) \subseteq \dom(\partial (h\circ A))$.
\end{proof}
\section{Final Remarks.}
\label{sec:conclusions}
We establish the linear convergence of the Douglas-Rachford algorithm for the problem $\min_x \{f(x)+g(x)\}$ and its Fenchel dual $\max_{u} \{ -f^*(u) - g^*(-u)\}\Leftrightarrow \min_{u} \{ f^*(u) + g_*(u)\}$ for convex functions~$f$ and~$g$ (Theorem~\ref{thm.DR}) via a generic error bound condition~\eqref{eq.error.bound.F}.  We subsequently leverage this error bound condition to show and estimate the linear convergence of the Douglas-Rachford algorithm  in terms of some suitable measure of well-posedness of the primal and dual problems for three classes of problems.  The first class is when $f$ or $g$ {\em and} $f^*$ or $g_*$ are strongly convex relative to the optimal primal and dual sets  (Theorem~\ref{thm.stronglyConv.bound}).  The second class is when both $f$ and $g$, or equivalently both $f^*$ and $g_*$, are piecewise linear-quadratic (Theorem~\ref{thm:HpieceWiseLin}).  The third class is the conic feasibility problem when the functions $f$ and $g$ are the indicator functions of closed convex cones (Theorem~\ref{prop.error.gral}).  We also establish stronger finite termination properties when the cones are a linear subspace and the non-negative orthant (Theorem~\ref{thm.finite.feas}). Our results in the special conic feasibility problem naturally suggest two directions for future work.  We conjecture that the finite termination property also holds when~$f$ and~$g$ are the indicator functions of closed convex cones as long as the relative interiors of the cones have a nonempty intersection.  We also conjecture that the finite termination result for the case of a  linear subspace and the non-negative orthant can be strengthened by showing that the number of iterations until termination depends on some suitable condition measure of the problem.


\begin{thebibliography}{}

\bibitem[Bauschke et~al., 2016]{bauschke2016slater}
Bauschke, H.~H., Dao, M.~N., Noll, D., and Phan, H.~M. (2016).
\newblock On {S}later's condition and finite convergence of the
  {D}ouglas--{R}achford algorithm for solving convex feasibility problems in
  {E}uclidean spaces.
\newblock {\em Journal of Global Optimization}, 65(2):329--349.

\bibitem[Boyd et~al., 2004]{boyd2004convex}
Boyd, S., Boyd, S.~P., and Vandenberghe, L. (2004).
\newblock {\em Convex optimization}.
\newblock Cambridge university press.

\bibitem[Boyd et~al., 2011]{boyd2011distributed}
Boyd, S., Parikh, N., and Chu, E. (2011).
\newblock {\em Distributed optimization and statistical learning via the
  alternating direction method of multipliers}.
\newblock Now Publishers Inc.

\bibitem[Dadush et~al., 2020]{dadush2020rescaling}
Dadush, D., V{\'e}gh, L.~A., and Zambelli, G. (2020).
\newblock Rescaling algorithms for linear conic feasibility.
\newblock {\em Mathematics of Operations Research}, 45(2):732--754.

\bibitem[Douglas and Rachford, 1956]{douglas1956numerical}
Douglas, J. and Rachford, H.~H. (1956).
\newblock On the numerical solution of heat conduction problems in two and
  three space variables.
\newblock {\em Transactions of the American mathematical Society},
  82(2):421--439.

\bibitem[Eckstein, 1989]{eckstein1989splitting}
Eckstein, J. (1989).
\newblock {\em Splitting methods for monotone operators with applications to
  parallel optimization}.
\newblock PhD thesis, Massachusetts Institute of Technology.

\bibitem[Eckstein and Bertsekas, 1992]{eckstein1992douglas}
Eckstein, J. and Bertsekas, D.~P. (1992).
\newblock On the {D}ouglas--{R}achford splitting method and the proximal point
  algorithm for maximal monotone operators.
\newblock {\em Mathematical Programming}, 55(1):293--318.

\bibitem[Giselsson and Boyd, 2016]{giselsson2016linear}
Giselsson, P. and Boyd, S. (2016).
\newblock Linear convergence and metric selection for {D}ouglas-{R}achford
  splitting and {ADMM}.
\newblock {\em IEEE Transactions on Automatic Control}, 62(2):532--544.

\bibitem[Goldman and Tucker, 1956]{goldman1956linear}
Goldman, A. and Tucker, A. (1956).
\newblock Linear inequalities and related systems.
\newblock {\em Theory of Linear Programming}, 38:53--97.

\bibitem[Gutman and Pe\~na, 2020]{gutman2020condition}
Gutman, D.~H. and Pe\~na, J.~F. (2020).
\newblock The condition number of a function relative to a set.
\newblock {\em Mathematical Programming}, pages 1--40.

\bibitem[Hoffman, 1952]{Hoff52}
Hoffman, A.~J. (1952).
\newblock On approximate solutions of systems of linear inequalities.
\newblock {\em Journal of Research of the National Bureau of Standards},
  49(4):263--265.

\bibitem[Hoffman, 2003]{hoffman2003approximate}
Hoffman, A.~J. (2003).
\newblock On approximate solutions of systems of linear inequalities.
\newblock In {\em Selected Papers Of Alan J Hoffman: With Commentary}, pages
  174--176. World Scientific.

\bibitem[Hong and Luo, 2017]{hong2017linear}
Hong, M. and Luo, Z.-Q. (2017).
\newblock On the linear convergence of the alternating direction method of
  multipliers.
\newblock {\em Mathematical Programming}, 162(1-2):165--199.

\bibitem[Ioffe, 2017]{ioffe2017variational}
Ioffe, A.~D. (2017).
\newblock Variational analysis of regular mappings.
\newblock {\em Springer Monographs in Mathematics. Springer, Cham}.

\bibitem[Lions and Mercier, 1979]{lions1979splitting}
Lions, P.-L. and Mercier, B. (1979).
\newblock Splitting algorithms for the sum of two nonlinear operators.
\newblock {\em SIAM Journal on Numerical Analysis}, 16(6):964--979.

\bibitem[Lucet et~al., 2009]{lucet2009piecewise}
Lucet, Y., Bauschke, H.~H., and Trienis, M. (2009).
\newblock The piecewise linear-quadratic model for computational convex
  analysis.
\newblock {\em Computational Optimization and Applications}, 43(1):95--118.

\bibitem[Parikh and Boyd, 2014]{parikh2014proximal}
Parikh, N. and Boyd, S. (2014).
\newblock Proximal algorithms.
\newblock {\em Foundations and Trends in optimization}, 1(3):127--239.

\bibitem[Pe\~na et~al., 2020]{pena2020new}
Pe\~na, J., Vera, J.~C., and Zuluaga, L.~F. (2020).
\newblock New characterizations of {H}offman constants for systems of linear
  constraints.
\newblock {\em Mathematical Programming}, pages 1--31.

\bibitem[Rockafellar, 1970]{rockafellar1970convex}
Rockafellar, R.~T. (1970).
\newblock {\em Convex {A}nalysis}, volume~36.
\newblock Princeton university press.

\bibitem[Rockafellar and Wets, 2009]{rockafellar2009variational}
Rockafellar, R.~T. and Wets, R. J.-B. (2009).
\newblock {\em Variational {A}nalysis}, volume 317.
\newblock Springer Science \& Business Media.

\bibitem[Spingarn, 1985]{spingarn1985primal}
Spingarn, J.~E. (1985).
\newblock A primal-dual projection method for solving systems of linear
  inequalities.
\newblock {\em Linear Algebra and its Applications}, 65:45--62.

\bibitem[Wang and Shroff, 2017]{wang2017new}
Wang, S. and Shroff, N. (2017).
\newblock A new alternating direction method for linear programming.
\newblock In {\em Proceedings of the 31st International Conference on Neural
  Information Processing Systems}, pages 1479--1487.

\bibitem[Zhou, 2018]{zhou2018fenchel}
Zhou, X. (2018).
\newblock On the {F}enchel duality between strong convexity and {L}ipschitz
  continuous gradient.
\newblock {\em arXiv preprint arXiv:1803.06573}.

\end{thebibliography}
\end{document}